\documentclass[10pt]{amsart}

\usepackage{amsmath,amsthm}
\usepackage{amssymb,latexsym}

\makeatletter
\@addtoreset{equation}{section}
\makeatother

\newtheorem{theorem}{Theorem}[section]
\newtheorem{proposition}[theorem]{Proposition}
\newtheorem{lemma}[theorem]{Lemma}
\newtheorem{corollary}[theorem]{Corollary}
\newtheorem{question}[theorem]{Question}

\theoremstyle{definition}

\newtheorem{example}[theorem]{Example}

\newcommand{\wbar}[1]{\overline{#1}}
\newcommand{\what}[1]{\widehat{#1}}

\newcommand{\til}[1]{\tilde{#1}}

\newcommand{\spn}{\operatorname{span}}
\newcommand{\supp}{\operatorname{supp}}

\newcommand{\re}{\operatorname{Re}}
\newcommand{\im}{\operatorname{Im}}

\newcommand{\fA}{\mathcal{A}}
\newcommand{\fB}{\mathcal{B}}
\newcommand{\fC}{\mathcal{C}}
\newcommand{\fD}{\mathcal{D}}
\newcommand{\fE}{\mathcal{E}}
\newcommand{\fF}{\mathcal{F}}
\newcommand{\fH}{\mathcal{H}}
\newcommand{\fM}{\mathcal{M}}
\newcommand{\fU}{\mathcal{U}}

\newcommand{\fa}{\mathfrak{a}}
\newcommand{\fe}{\mathfrak{e}}
\newcommand{\ff}{\mathfrak{f}}
\newcommand{\fg}{\mathfrak{g}}
\newcommand{\fh}{\mathfrak{h}}
\newcommand{\fk}{\mathfrak{k}}
\newcommand{\fm}{\mathfrak{m}}
\newcommand{\fn}{\mathfrak{n}}
\newcommand{\fr}{\mathfrak{r}}
\newcommand{\fs}{\mathfrak{s}}
\newcommand{\fsu}{\mathfrak{su}}

\newcommand{\su}{\mathrm{SU}}

\newcommand{\ad}{\operatorname{ad}}
\newcommand{\diag}{\operatorname{diag}}

\newcommand{\Cee}{\mathbb{C}}
\newcommand{\Hee}{\mathbb{H}}
\newcommand{\Ree}{\mathbb{R}}
\newcommand{\Tee}{\mathbb{T}}
\newcommand{\Zee}{\mathbb{Z}}
\newcommand{\En}{\mathbb{N}}

\newcommand{\Ess}{\mathbb{S}}

\newcommand{\alp}{\alpha}
\newcommand{\del}{\delta}
\newcommand{\Del}{\Delta}
\newcommand{\eps}{\varepsilon}
\newcommand{\gam}{\gamma}
\newcommand{\Gam}{\Gamma}
\newcommand{\lam}{\lambda}
\newcommand{\Lam}{\Lambda}

\newcommand{\sig}{\sigma}

\newcommand{\vphi}{\varphi}

\newcommand{\tr}{\mathrm{Tr}}

\newcommand{\norm}[1]{\left\Vert#1\right\Vert}

\begin{document}

\title[Weak amenability of $A(G)$]
{Weak amenability of Fourier algebras and local synthesis of the anti-diagonal}

\author{Hun Hee Lee, Jean Ludwig, Ebrahim Samei and Nico Spronk}

\begin{abstract}
We show that for a connected Lie group $G$, its Fourier algebra $A(G)$ is weakly amenable
only if $G$ is abelian.
Our main new idea is to show that  weak amenability of $A(G)$ implies that
the anti-diagonal, $\check{\Del}_G=\{(g,g^{-1}):g\in G\}$, is a set of local synthesis for
$A(G\times G)$.  We then show that this cannot happen if $G$ is non-abelian.
We conclude for a locally compact group $G$, that $A(G)$ can be weakly amenable only if
it contains no closed connected non-abelian Lie subgroups. In particular, for a Lie group $G$,
$A(G)$ is weakly amenable if and only if its connected component of the identity $G_e$ is abelian.
\end{abstract}

\maketitle

\footnote{{\it Date}: \today.


The first named author was supported by the Basic Science Research Program through the National Research Foundation of Korea (NRF), grant  NRF-2015R1A2A2A01006882.
The second named author was supported by Institut \'{E}lie Cartan de Lorraine.
The third named author was supported by NSERC Grant 366066-2014, and the
Professeur Invit\'{e} program at Institut \'{E}lie Cartan de Lorraine.
The fourth named was supported by NSERC Grant 312515-2010, and the Korean Brain Pool Program.}


\subsection{Background}  Questions on the nature of bounded derivations on (commutative)
Banach algebras $\fA$ have been around for a long time, in particular vanishing
of bounded Hochschild cohomologies $H^1(\fA,\fM)$ for certain Banach $\fA$-modules $\fM$.
See, for example, \cite{singerw,kamowitz}.  Johnson systematized many
of these questions in \cite{johnsonM}.  In particular, he showed that for a locally compact group $G$,
its group algebra is {\it amenable} (i.e.\ $H^1(L^1(G),\fM^*)=\{0\}$ for each dual module $\fM^*$)
if and only if $G$ is an amenable group.  He also started the problem of determining when
$H^1(L^1(G),L^1(G)^*)=\{0\}$.

For a commutative Banach algebra $\fA$, Bade, Curtis and Dales (\cite{badecd})
introduced the concept of {\it weak amenability}, which is defined as having $H^1(\fA,\fM)=\{0\}$ for all
symmetric Banach modules.  They observed that this is equivalent to having
$H^1(\fA,\fA^*)=\{0\}$.  There is an interesting universal module also exhibited by Runde
(\cite{runde}).  The above observation of \cite{badecd}, leads
us to refer to any Banach algebra $\fB$ as {\it weakly amenable} if $H^1(\fB,\fB^*)=\{0\}$.
Weak amenability was established for all $L^1(G)$ by Johnson (\cite{johnson0}).

The Fourier algebras, $A(G)$, as defined by Eymard (\cite{eymard}), are dual objects
to the group algebras $L^1(G)$ in a sense which generalizes Pontryagin duality.
It was long expected that the amenability properties enjoyed by group algebras
would also extend to Fourier algebras.  Hence it was a surprise when Johnson (\cite{johnson})
showed that $A(G)$ fails to be weakly amenable for any compact simple Lie group,
in particular for $G=\mathrm{SO}(3)$.  This motivated Ruan (\cite{ruan})
to consider the operator space structure $A(G)$ inherits by virtue of being the predual of a von Neuman algebra.  He proved that $A(G)$ is operator amenable if and only if $G$ is amenable.
Operator weak amenability for general $A(G)$ was determined by Spronk (\cite{spronk})
and, independently, by Samei (\cite{samei}).   The question of amenability of $A(G)$ was settled
by Forrest and Runde (\cite{forrestr}): it happens exactly when $G$ is virtually abelian.
They also showed that if the connected component $G_e$ is abelian, then
$A(G)$ is weakly amenable.  The following is suggested.

\begin{question}\label{ques:forrest}
If $A(G)$ is weakly amenable, then must $G_e$ be abelian?
\end{question}

Much progress has been made in answering this question.  Building on work of Plymen
(\cite{plymen}) -- which was written to answer a question in \cite{johnson} -- Forrest, Samei and Spronk
(\cite{forrestss1}) showed that $A(G)$ is not weakly amenable whenever $G$ contains
a non-abelian connected compact subgroup.  Exciting recent progress was made by
Choi and Ghandehari.  In \cite{choig} they show for the affine motion group, and hence
any simply connected semisimple Lie group, and also for the reduced Heisenberg group,
that the Fourier algebra is not weakly amenable.  In \cite{choig1} they used completely different
techniques to show the same for Heisenberg groups.  Our main theorem generalizes all of these
results.

Let us briefly review the history of ideas around spectral synthesis.  All concepts below will be defined
in Section \ref{ssec:WALSAD}.  The study of sets of spectral
synthesis, or, for us, simply ``synthesis", for $A(G)$, especially for abelian $G$, has a long history.
See, for example, the historical notes in \cite[\S 42]{hewittrII}.  Herz (\cite{herz}) appears to have been the
first author to consider local synthesis for general $G$.  There, he proved this property is enjoyed
by closed subgroups.  Herz's work has inspired,  in part, the injection theorem of
Lohou\'{e} (\cite{lohoue}) and has motived aspects of the work of Ludwig and Turowska (\cite{ludwigt0}).
Thanks to the existence of a bounded approximate identity consisting of compactly supported functions
(\cite{leptin}), sets of local synthesis are sets of synthesis when $G$ is amenable.
Weak synthesis has its origins in work of Varopoulos (\cite{varopoulos}) on spheres in $\Ree^n$
($n\geq 3$), was used in Kirsch and M\"{u}ller (\cite{kirsch}),  and was formalized by
Warner (\cite{warner}).  The first explicit mention of what we call smooth synthesis is by
M\"{u}ller (\cite{muller}), which was applied to certain manifolds in $\Ree^n$.  The first
use of this concept for non-abelian $G$ is due to Ludwig and Turowska (\cite{ludwigt}).
Following their work, Park and Samei (\cite{parks}) showed that for a connected Lie group $G$,
the anti diagonal $\check{\Del}_G$ is a set of smooth synthesis, and also of weak
synthesis for $A(G\times G)$.

For compact $G$, building on work of Gr{\o}nb{\ae}k (\cite{groenbaek}), the article \cite{johnson} used
the failure of a weak form of spectral synthesis of the diagonal $\Del_G=\{(g,g):g\in G\}$
for the projective tensor product algebra $A(G)\hat{\otimes}A(G)$, to obtain the failure of weak amenability.
The local synthesis of $\Del_G$ for general locally compact $G$
was used by Samei (\cite{samei1}) to study a property
which implies weak amenability of $A(G)$.
Returning to compact groups, the ideas of \cite{johnson} were formalized
and capitalized upon in \cite{forrestss2}.  These were used in  \cite{forrestss1}
to show that for compact $G$, $A(G)$ is weakly amenable exactly when
the anti-diagonal $\check{\Del}_G=\{(g,g^{-1}):g\in G\}$ is a set of synthesis
for $A(G\times G)$. For groups containing open subgroups
products of abelian groups and compact groups, i.e.\ $G\supseteq H\cong A\times K$,
this last result was extended to local synthesis.
We recall that $A(G)\hat{\otimes}A(G)\not=A(G\times G)$,
generally (\cite{losert}).   Hence we do not expect obvious connections
between sets of local synthesis for these two algebras.

\subsection{Structure.}  
The starting point for the present investigation lies in the aforementioned results
of \cite{parks}.  Let $G$ be a connected Lie group.
We use the fact that $\check{\Del}_G$ is simultaneously of weak and smooth synthesis for
$A(G\times G)$, along with the characterization of weak amenability
of \cite{runde}, to show that for a connected Lie group, weak amenability
of $A(G)$ implies that $\check{\Del}_G$ is a set of local synthesis for $A(G\times G)$.
The techniques rely intrinsically on the Lie theory, especially having finite dimension for $G$.
We see no way, at present, to extend them to arbitrary connected locally compact groups.

In Section \ref{sec:main}, we show for any connected Lie $G$ that that weak amenability of $A(G)$
implies local synthesis of $\check{\Del}_G$ for $A(G\times G)$ -- a property
we shall hereafter call ``local synthesis for $G\times G$''.
Also in that section we discuss our two main functorial properties satisfied by local synthesis
of $\check{\Del}_G$ for $G\times G$ for locally comapct $G$: the restriction to a closed connected Lie subgroup
and an injection theorem with quotients by discrete normal subgroups.
In section \ref{sec:fivegroups} we give a criterion for testing local synthesis of the anti-diagonal, and
we show for five (classes of) low-dimensional Lie groups that this criterion is satisfied.
In Section \ref{sec:final} we tie the investigation together by noting that
any non-abelian connected Lie group contains one of the five aforementioned groups or its simply
connected covering group, and use the functorial properties to draw our conclusion for connected $G$.
We can thus answer Question \ref{ques:forrest} affirmatively for all Lie groups. 

\subsection{Basic notation.}  Let $G$ be a locally compact group.
The following spaces will be used in this note:  the space $\fC_c(G)$ of compactly
supported continuous functions; and the $L^p$-spaces with respect the
left Haar measure, $L^p(G)$, $p=1,2,\infty$.

We follow Eymard (\cite{eymard}) for all definitions and concepts around the
Fourier algebra $A(G)$.  We recall that $A(G)$ consists of all matrix coefficients
$u(g)=\langle\lam(g)\xi|\eta\rangle$ where $\lam:G\to \fU(L^2(G))$ is the left
regular representation, and $\xi,\eta\in L^2(G)$.  Furthermore, the norm is given by
$\norm{u}_A=\inf\{\norm{\xi}_2\norm{\eta}_2:u=\langle\lam(\cdot)\xi|\eta\rangle\text{ as above}\}$.
The bounded linear dual
is given by the group von Neuman algebra $VN(G)=\lam(G)''\subset\fB(L^2(G))$.
If $u=\langle\lam(\cdot)\xi|\eta\rangle\in A(G)$ and $T\in VN(G)$ we write
$T(u)=\langle T\xi|\eta\rangle$.
We shall denote the operator norm of an element $S$ of $VN(G)$ by $\norm{S}_{VN}$.
We recall that the maps $u\mapsto\check{u}$ ($\check{u}(g)
=u(g^{-1})$) and maps of left and right translation are all isometric automorphisms
of $A(G)$.  We always let $A(G)\hat{\otimes}A(G)$
denote the projective tensor product.  We shall have no need
for completely bounded maps, and will make no use of the operator projective tensor product,
except, of course, implicitly, when we discuss $A(G\times G)$.

We shall define more specialized notions in situ as their necessities arise.

\section{Weak amenability and local synthesis for the anti-diagonal}\label{sec:main}

\subsection{Weak amenability implies local synthesis for the anti-diagonal}\label{ssec:WALSAD}
Let $G$ be a locally compact group.  We let $A_c(G)$ denote the subalgebra of elements
$u$ of $A(G)$ for which $\supp u=\wbar{\{g\in G:u(g)\not=0\}}$ is compact.
It is well known that the {\it Tauberian condition} holds:  $A_c(G)$ is dense in $A(G)$; and
that $A(G)$ is {\it regular} on $G$:  given compact $K$ and a neighbourhood $U$ of $K$,
there is $u$ in $A_c(G)$ for which $u|_K=1$ and $\supp u\subset U$.
See \cite[(3.38) and (3.2)]{eymard}.

Given a closed subset $E$ of $G$ we let
\begin{align*}
I_G(E)&=\{u\in\ A(G):u|_E=0\} \\
I_G^0(E)&=\{u\in\ A_c(G):\supp u\cap E=\varnothing\}\text{, and}\\
J_G(E)&=\wbar{I_G(E)\cap A_c(G)}.
\end{align*}
It is evident that $\wbar{I_G^0(E)}\subseteq J_G(E)\subseteq I_G(E)$.  We say
that $E$ is a set of

$\bullet$ {\it spectral synthesis} for $G$ if $\wbar{I_G^0(E)}= I_G(E)$; and

$\bullet$ {\it local synthesis} for $G$ if $\wbar{I_G^0(E)}= J_G(E)$.

\noindent
For a linear subspace $S\subset A(G)$  we let $S^{(d)}=\spn\{u^d:u\in S\}$.  We say that $E$ is a set of

$\bullet$ {\it weak synthesis} for $G$ if there is $d$ in $\En$ for which
$I_G(E)^{(d)}\subseteq\wbar{I_G^0(E)}$; and

$\bullet$ {\it local weak synthesis} for $G$ if there is $d$ in $\En$ for which
$J_G(E)^{(d)}\subseteq\wbar{I_G^0(E)}$.

We now let $G$ be a connected Lie group  and
let $\fD(G)$ denote the space of compactly supported smooth functions on $G$.
If $K$ is a compact subset of $G$ we let $\fD_K(G)=\{u\in\fD(G):\supp u\subset K\}$,
which is a Fr\'{e}chet space.  For example, we fix a basis $\beta=(X_1,\dots,X_d)$ for the Lie
algebra $\fg$ of $G$ and for $u,v$ in $\fD_K(G)$ set
\begin{gather}
\partial_{X}u(g)=\left.\frac{d}{dt}u(g\exp(tX))\right|_{t=0},\text{ for }X\text{ in }\fg \label{eq:der} \\
\rho_{\beta,n}^K(u)=\sum_{1\leq i_1,\dots,i_n\leq d}
\norm{ \partial_{X_{i_1}}\dots\partial_{X_{i_n}}u }_\infty \notag \\
\text{and }\rho_\beta^K(u,v)=\sum_{n=0}^\infty\frac{\rho_{\beta,n}^K(u-v)}{2^n+\rho_{\beta,n}^K(u-v)}.\notag
\end{gather}
Then given a compact neighbourhood $K$ of the identity,
$\fD(G)=\bigcup_{n=1}^\infty \fD_{K^n}(G)$
is an inductive limit of Fr\'{e}chet spaces, a so called {\it LF-space}.  See, for example, \cite{treves}.
It follows from \cite[(3.26)]{eymard} or
\cite[Lemma 3.3 and Remark 4.2]{ludwigt}, that $\fD(G)\subset A_c(G)$,
and each inclusion $\fD_K(G)\hookrightarrow A(G)$ is continuous.  Furthermore, since
$\fD(G)$ is dense in $L^2(G)$, $\fD(G)\supseteq \fD(G)\ast\fD(G)$ and
$\{\check{u}:u\in\fD(G)\}=\fD(G)$, $\fD(G)$
is dense in $A(G)$.

For a closed subset $E$ of $G$ we let
\[
J^\fD_G(E)=\wbar{\fD(G)\cap I_G(E)}
\]
where closure is in $A(G)$.  Since $\fD(G)$ is dense in $A(G)$, and $\fD(G)J^\fD_G(E)\subseteq
J^\fD_G(E)$, we have that $J^\fD_G(E)$ is an ideal in $A(G)$.  Since $A(G)$ is regular and Tauberian
we have inclusions $I^0_G(E)\subseteq J^\fD_G(E)\subseteq J_G(E)\subseteq I_G(E)$.
We say that $E$ is a set of

$\bullet$ {\it smooth synthesis} for $G$ if $J^\fD_G(E)= I_G(E)$; and

$\bullet$ {\it local smooth synthesis} for $G$ if $J_G^\fD(E)= J_G(E)$.

\noindent The projective tensor product $\fE\hat{\otimes}\fF$ of two locally convex spaces $\fE$ and $\fF$
is the completion of the algebraic tensor product $\fE\otimes\fF$ in the final topology
with respect to the embedding $\fE\times\fF\hookrightarrow\fE\otimes\fF$.  The standard
proof of the following,  for $G=\Ree^d$,
uses techniques of Fourier analysis on Schwarz class functions.  Hence we need
to take a little care to see that it holds in our setting, though the proof is standard.

\begin{lemma}\label{lem:tptf}
Let $G$ be a connected Lie group and $K$ and $M$ be compact subsets of
$G$.  Then $\fD_K(G)\hat{\otimes}\fD_M(G)\cong\fD_{K\times M}(G\times G)$ linearly
and homeomorphically.
\end{lemma}

\begin{proof}
Let us first assume that there is a neighbourhood $U$ of $K$ for which there
is a diffeomorphism $\vphi:U\to U'\subset\Ree^d$.  Then the map
$u\mapsto u\circ\vphi:\fD_K(G)\to\fD_{\vphi(K)}(\Ree^d)$
is a linear homeomorphism.  The same fact holds for $M$,
with a neighbourhood $V$ of $M$ and a diffeomorphism $\psi:V\to V'\subset\Ree^n$.
Thus we get linear homeomorphisms
\[
\fD_K(G)\hat{\otimes}\fD_M(G)\cong\fD_{\vphi(K)}(\Ree^d)\hat{\otimes}\fD_{\psi(M)}(\Ree^d)
\cong \fD_{\vphi(K)\times\psi(M)}(\Ree^{2d})\cong\fD_{K\times M}(G\times G)
\]
where the middle identification is provided by \cite[Theorem 51.6]{treves}, and the last
one by the map $w\mapsto w\circ(\vphi^{-1}\times\psi^{-1})$.

Generally, there are finite open covers $\{U_i\}_{i=1}^m$ of $K$ and
$\{V_j\}_{j=1}^n$ of $M$, such that each member is diffeomorphic to an open subset of $\Ree^n$.
Let $\{u_i\}_{i=1}^m,\{v_j\}_{j=1}^n$ in $\fD(G)$ be smooth partitions of unity, subordinate
to the respective covers.  Then we obtain, for example, a linear homeomorphism
$u\mapsto \sum_{i=1}^n uu_i:\fD_K(G)\to\sum_{i=1}^n\fD_{\wbar{K\cap U_i}}(G)$, whose
inverse is mere inclusion.  Then we obtain linear homeomorphisms
\[
\left(\bigoplus_{i=1}^m\fD_{\wbar{K\cap U_i}}(G)\right)\hat{\otimes}
\left(\bigoplus_{j=1}^n\fD_{\wbar{M\cap V_j}}(G)\right)\cong
\bigoplus_{i=1}^m\bigoplus_{j=1}^n\fD_{\wbar{K\cap U_i}\times\wbar{M\cap V_j}}(G\times G)
\]
as above.  Hence we obtain linear homeomorphisms
\begin{align*}
\fD_K(G)\hat{\otimes}\fD_M(G)&\cong\sum_{i=1}^m\sum_{j=1}^n
\fD_{\wbar{K\cap U_i}}(G)\hat{\otimes}\fD_{\wbar{M\cap V_j}}(G) \\
&\cong\sum_{i=1}^m\sum_{j=1}^n\fD_{\wbar{K\cap U_i}\times\wbar{M\cap V_j}}(G\times G)
\cong \fD_{K\times M}(G\times G)
\end{align*}
by using both injectivity and projectivity of tensor product of these nuclear spaces.
\end{proof}

We let $\check{\Del}_G=\{(g,g^{-1}):g\in G\}\subset G\times G$.

\begin{lemma}\label{lem:algtensapprox}
Let $G$ be a connected Lie group.  Then
\[
(A_c(G)\otimes A_c(G))\cap J_{G\times G}(\check{\Del}_G)
\]
is dense in $J_{G\times G}(\check{\Del}_G)$.
\end{lemma}

\begin{proof}
For simplicity, we let $A=A(G\times G)$. Let $u\in J_{G\times G}(\check{\Del}_G)$.  It was shown in \cite[Theorem 7]{parks}
that $\check{\Del}_G$ is a set of local smooth synthesis, so we may assume that $u\in\fD(G\times G)$.
Hence we can find a compact subset $M$ of $G$ for which $u\in\fD_{M\times M}(G\times G)$.
We can further assume that $M$ is symmetric:  $M^{-1}=M$.  Lemma \ref{lem:tptf} provides
that $u\in\fD_M(G)\hat{\otimes}\fD_M(G)$.

Fix $\eps>0$.
Let $\rho$ be any invariant metric on $\fD_{M\times M}(G\times G)$ which comes form the Fr\'{e}chet
structure.  There is a $\del_1>0$ for which
\[
\rho(w,0)< \del_1 \quad\Rightarrow\quad \norm{w}_A<\eps
\]
for $w\in \fD_{M\times M}(G\times G)\subset A(G\times G)$.  Furthermore, it is straightforward to check that the map
\[
\Lam:\fD_{M\times M}(G\times G)\to\fD_M(G)\text{ given by }\Lam w(g)=w(g,g^{-1})
\]
is continuous since $\check{\Del}_G$, qua submanifold of $G\times G$, is diffeomorphic to $G$.
If we fix $v_M$ in $A_c(G)$ for which $v_M|_M=1$, there is $\del_2>0$ for which
\[
\rho(w,0)< \del_2\quad\Rightarrow\quad \norm{\Lam(w)}_{A(G)}<\frac{\eps}{\norm{v_M}_{A(G)}}.
\]
Thus, if $\del=\min\{\del_1,\del_2\}$, then given $u$ as in the paragraph above,
there is a $v$ in the algebraic tensor product $\fD_M(G)\otimes\fD_M(G)$ for which
$\rho(u,v)<\del$, and hence our choice of $\del$ entails that
\[
\norm{u-v}_A<\eps\text{ and }
\norm{\Lam(u-v)}_{A(G)}<\frac{\eps}{\norm{v_M}_{A(G)}}.
\]
Now let $w=v-\Lam(v)\otimes v_M$ which is an element of $A_c(G)\otimes A_c(G)$.
Since $M=M^{-1}$ it is easy to see that $w|_{\check{\Del}_G}=0$.  Moreover, notice
that $\Lam(u)=0$, as $u\in J^\fD_{G\times G}(\check{\Del}_G)$, so
\[
\norm{u-w}_A\leq \norm{u-v}_A+\norm{\Lam(u-v)\otimes v_M}_{A}
< 2\eps
\]
since $\norm{\Lam(u-v)\otimes v_M}_{A}\leq\norm{\Lam(u-v)}_{A(G)}\norm{v_M}_{A(G)}$.
\end{proof}

We let $A(G)^\sharp$ denote the unitization of $A(G)$ and 
$m^\sharp:A(G)^\sharp\hat{\otimes}A(G)^\sharp\to A(G)^\sharp$ and
$m:A(G)\hat{\otimes}A(G)\to A(G)$ be the continuous linearization of the respective
multiplication maps.  Since $A(G)$ is Tauberian and regular, 
$\wbar{A(G)^2}=A(G)$.  Hence
It follows from \cite[Theorem 3.2]{groenbaek}  that $A(G)$ is weakly amenable if and only if
\begin{equation}\label{eq:Kideal}
\wbar{(\ker m)^2}=\wbar{A(G)\hat{\otimes}A(G)\cdot \ker m^\sharp}.
\end{equation}
Above, and in what follows, we use for a subspace $S$ of an algebra $A$, the notation
$S^d=\spn\{u_1\dots u_d:u_1,\dots,u_d\in S\}$.

\begin{theorem}\label{theo:wassad}
Let $G$ be a connected Lie group.  If $A(G)$ is weakly amenable, then
$\check{\Del}_G$ is a set of local synthesis for $A(G\times G)$.
\end{theorem}

\begin{proof}
We let $\check{m}:A(G)\hat{\otimes}A(G)\to A(G)$ be given on elementary tensors
by $\check{m}(u\otimes v)=u\check{v}$, and likewise define $\check{m}^\sharp$.
Since $v\mapsto\check{v}$ is an isometry on $A(G)$, if $A(G)$ is weakly amenable
then (\ref{eq:Kideal}) implies that
\begin{equation}\label{eq:cKideal}
\wbar{(\ker\check{m})^2}=\wbar{A(G)\hat{\otimes}A(G)\cdot\ker \check{m}^\sharp}.
\end{equation}
For simplicity, we write $J=J_{G\times G}(\check{\Del}_G)$, below.
Let $u\in J$.  We wish to approximate
$u$ by elements from $I^0_{G\times G}(\check{\Del}_G)$.
We let $\iota:A(G)\hat{\otimes}A(G)\to A(G\times G)$ be the
linear contraction which embeds $A(G)\otimes A(G)$ into $A(G\times G)$.
We have that
\[
[A_c(G)\otimes A_c(G)]\cap J
\subset\iota(A_c(G)\otimes A_c(G)\cdot\ker \check{m}^\sharp)
\subseteq\wbar{\iota(\ker \check{m})^2}\cap A_c(G\times G)\subseteq\wbar{J^2}
\]
where the first inclusion follows from regularity of $A(G)$, the second inclusion is 
provided by (\ref{eq:cKideal}), and the third inclusion follows from regularity of $A(G\times G)$.
By Lemma \ref{lem:algtensapprox},  $[A_c(G)\otimes A_c(G)]\cap J$ is dense in $J$.
We thus conclude that $J=\wbar{J^2}$, and induction shows that $J=\wbar{J^m}$
for any $m$ in $\En$.

The identity $4uv=(u+v)^2-(u-v)^2$ shows that
$J^2=J^{(2)}$, where the latter notation was used in the definition of weak synthesis, above.
By induction we see that $J^{2^m}=J^{(2^m)}$ for any $m$ in $\En$.  We conclude that
$J=\wbar{J^{(2^m)}}$ for any $m$ in $\En$.

It is shown in \cite[Theorem 7]{parks}, that $\check{\Del}_G$ is a set of local weak synthesis, i.e.\
if $n\geq\dim(G)/2$, then $J^{(n)}\subset I^0_{G\times G}(\check{\Del}_G)$.  Hence we see
that for some $m$, $J=\wbar{J^{(2^m)}}=\wbar{I^0_{G\times G}(\check{\Del}_G)}$, and
local synthesis is established.
\end{proof}

\subsection{Functorial properties for local synthesis of the anti-diagonal}
Our ultimate goal is to show that a non-abelian connected Lie group
does not allow local synthesis of the anti-diagonal. The following
will allow us to reduce our calculations to certain computable cases.

We shall make use of the well-known result that for any locally compact
group $G$, and closed subgroup $H$, the restriction map
$R_H:A(G)\to A(H)$ is a quotient map.  See \cite{herz},  \cite[(4.21)]{mcmullen},
\cite[(3.23)]{arsac} or \cite{delaported}.  We even have $R_H(A_c(G))=A_c(H)$.

\begin{theorem}\label{theo:restriction}
Let $G$ be a locally compact group and $H$ is a closed connected Lie subgroup.
If $\check{\Del}_G$ is of local synthesis for $G\times G$, then $\check{\Del}_H$
is of local synthesis for $H\times H$.
\end{theorem}

\begin{proof}
We first claim that
\[
\wbar{R_{H\times H}(J_{G\times G}(\check{\Del}_G))}=J_{H\times H}(\check{\Del}_H).
\]
It is clear that $R_{H\times H}(J_{G\times G}(\check{\Del}_G))\subseteq J_{H\times H}(\check{\Del}_H)$.
To see the converse inclusion, let $u\in J_{H\times H}(\check{\Del}_H)$.  Since
$H$ is a connected Lie group, Lemma
\ref{lem:algtensapprox} allows us to assume that $u\in A_c(H)\otimes A_c(H)$, hence
$u=\sum_{i=1}^n u_i\otimes v_i$.  The restriction theorem assures that there are
elements $u_i',v_i'$ in $A_c(G)$ for which $R_Hu_i'=u_i$, $R_Hv_i'=v_i$ for each $i$.
Hence if $w$ in $A_c(G)$ satisfies that $w|_S=1$ where $S=\bigcup_{i=1}^n\supp(u_i)$, then
\[
u'=\sum_{i=1}^n(u_i'\otimes v_i'-w\otimes\check{u}_i'v_i')\in I_{G\times G}(\check{\Del}_G)\cap A_c(G\times G)
\]
and, since $\sum_{i=1}^n\check{u}_i(h)v_i(h)=u(h^{-1},h)=0$ for $h$ in $H$, we have
\[
R_{H\times H}(u')=\sum_{i=1}^n(u_i\otimes v_i-R_H(w)\otimes\check{u}_iv_i)=u.
\]
It then follows that $u\in R_{H\times H}(J_{G\times G}(\check{\Del}_G))$.

It is evident that
\[
R_{H\times H}(I^0_{G\times G}(\check{\Del}_G))\subseteq
I^0_{H\times H}(\check{\Del}_H),\quad\text{so}\quad
R_{H\times H}\left(\wbar{I^0_{G\times G}(\check{\Del}_G)}\right)\subseteq
\wbar{I^0_{H\times H}(\check{\Del}_H)}.
\]
Our assumption that $J_{G\times G}(\check{\Del}_G)=\wbar{I^0_{G\times G}(\check{\Del}_G)}$,
coupled with the result of the prior paragraph, shows that
$\wbar{I^0_{H\times H}(\check{\Del}_H)}=J_{H\times H}(\check{\Del}_H)$.
\end{proof}

The proof of the next functorial property is more general, though its proof is a little more involved.
We shall use the following fact about local synthesis.
This is related to well-known facts about spectral synthesis
in regular function algebras; see, for example, expositions in \cite{reitersB,kaniuthB}.
If $G$ is a locally compact group and $E$ is a closed subset of $G$, we have
for $u$ in $A_c(G)$ that
\[
u\in \wbar{I^0_G(E)}\quad\Leftrightarrow\quad
\begin{matrix}u\text{ is ``locally in }\wbar{I^0_G(E)}\text{", i.e.\ for every }g\text{ in }E\text{ there is a
neigh-\phantom{m}} \\
\text{bourhood }
U_g\text{ of }g\text{, and a }u_g\text{ in }\wbar{I^0_G(E)}
\text{, for which }u|_{U_g}=u_g|_{U_g}.
\end{matrix}
\]
Indeed, we need to only prove sufficiency.  In this case, take such a collection
$U_1=U_{g_1},\dots,U_n=U_{g_n}$ of open sets covering $\supp(u)\cap E$, and
associated functions $u_1,\dots,u_n$.  Let $U_{n+1}$
be a neighbourhood of $\supp(u)\setminus\bigcup_{k=1}^nU_k$ with
$U_{n+1}\cap E=\varnothing$.  Any partition of unity $v_1,\dots,v_{n+1}$
of $\supp(u)$ subordinate to $U_1,\dots,U_{n+1}$, satisfies $uv_k=u_kv_k$, for each $k=1,\dots,n$,
so $u=\sum_{k=1}^nu_kv_k+uv_{n+1}\in \wbar{I^0_G(E)}$.

The following holds for any regular  Banach function algebra.  To avoid introducing new
notation, we state it only for a Fourier algebra.

\begin{proposition}\label{prop:locsynthdisc}
Let $G$ be a locally compact group and $E$ be a closed subset  which
admits a partition, $E=\bigsqcup_{i\in I}E_i$, for which each $E_i$ is closed in $G$
and relatively open in $E$.  Then $E$ is of local synthesis for $G$ if and only if
each $E_i$ is of local synthesis.
\end{proposition}

\begin{proof}
($\Rightarrow$) Let $F$ be a subset of $E$ which is closed in $G$ and relatively open
in $E$;  in particular we may consider $F=E_i$ for a fixed $i$.
Let $u\in J_G(F)\cap A_c(G)$.
For every $g$ in $F$,  our assumptions on $F$  allows
us to choose a compact neighbourhood $U$ of $g$
for which $U\cap E=U\cap F$.
We find $v$ in $A_c(G)$ for which
$\supp(v)\cap E=\supp(v)\cap F$ and $v|_U=1$.
We have that $u=uv$ on $U$ and
\[
uv\in J_G(E)=
\wbar{I^0_G(E)}\subseteq\wbar{I^0_G(F)}.
\]
Thus $u$ is locally in $\wbar{I^0_G(F)}$, whence in $J_G(F)$.

($\Leftarrow$)  Let $u\in J_G(E)\cap A_c(G)$.
Since each $E_i$ is of local synthesis, we have
\[
u\in J_G(E)\subseteq J_G(E_i)=\wbar{I^0_G(E_i)}
\]
i.e.\ $u=\lim_{n\to\infty}u_n$, where each $u_n\in I^0_G(E_i)$.
Fix $g$ in $E_i$ and, again, our assumption of relative openness of $E_i$ in $E$ allows us to
take a compact neighbourhood $U$ of $g$ for which
$U\cap E=U\cap E_i$.  If $v$ is any element of $A_c(G)$ for which
$\supp(v)\cap E=\supp(v)\cap E_i$ and $v|_U=1$, then
$u=uv$ on $U$, and $uv=\lim_{n\to\infty}u_nv$, where each $u_nv\in I^0_G(E)$.
Hence $u$ is locally in $\wbar{I^0_G(E)}$, thus $u\in\wbar{I^0_G(E)}$.
\end{proof}

\begin{theorem}\label{theo:discnorm}
Let $G$ be a locally compact group and $\Gam$ a discrete normal subgroup.
Then $\check{\Del}_G$ is of local synthesis for $G\times G$ if and only if
$\check{\Del}_{G/\Gam}$ is of local synthesis for $G/\Gam\times G/\Gam$.
\end{theorem}

\begin{proof}
The injection theorem of  \cite{lohoue} tells us that
$\check{\Del}_{G/\Gam}$ is of local synthesis for $G/\Gam\times G/\Gam$
if and only if $q^{-1}(\check{\Del}_{G/\Gam})$, where $q:G\times G\to G/\Gam\times G/\Gam$
is the quotient map, is a set of local synthesis for $G\times G$.
Hence we wish to verify that
\begin{equation}\label{eq:injection}
\check{\Del}_G\text{ is of local synthesis}\quad\Leftrightarrow\quad
q^{-1}(\check{\Del}_{G/\Gam})\text{ is of local synthesis}
\end{equation}
each for $G\times G$.  Hence we wish to examine the structure of the latter set.

We first observe that since $g\Gam=\Gam g$ and $(\Gam g)^{-1}=g^{-1}\Gam$ in $G/\Gam$, we have
\[
q^{-1}(\check{\Del}_{G/\Gam})=\{(\gam g,g^{-1}{\gam'}^{-1}):g\in G,\gam,\gam'\in\Gam\}.
\]
Hence if we define an action of $\Gam\times\Gam$ on $G\times G$ by $(\gam,\gam')\cdot (g,g')
=(\gam g,g'{\gam'}^{-1})$, then $q^{-1}(\check{\Del}_{G/\Gam})=(\Gam\times\Gam)\cdot \check{\Del}_G$,
the orbit of the set $\check{\Del}_G$ under this action.  It is easy to check that
\[
(\gam,\gam')\cdot \check{\Del}_G= \check{\Del}_G
\quad\Leftrightarrow\quad \gam'=\gam
\]
i.e.\ if $(\gam g,g^{-1}{\gam'}^{-1})=(h,h^{-1})$ then $\gam g=(g^{-1}{\gam'}^{-1})^{-1}$;
and hence
\[
(\gam,\gam')\cdot \check{\Del}_G= (\lam,\lam')\cdot\check{\Del}_G
\quad\Leftrightarrow\quad \lam^{-1}\gam={\lam'}^{-1}\gam'
\quad\Leftrightarrow\quad \lam'\lam^{-1}=\gam'\gam^{-1}.
\]
Thus $q^{-1}(\check{\Del}_{G/\Gam})=\bigsqcup_{(\gam,\gam')
\Del_\Gam\in(\Gam\times\Gam)/\Del_\Gam}(\gam,\gam')\cdot \check{\Del}_G$,
where $\Del_\Gam=\{(\gam,\gam):\gam\in\Gam\}$.  Let us see that
\begin{equation}\label{eq:open}
\text{the individual fibres }(\lam,\lam')\cdot \check{\Del}_G\text{ are relatively
open in }q^{-1}(\check{\Del}_{G/\Gam}).
\end{equation}
Suppose that $(\lam h,h^{-1}{\lam'}^{-1})$
is approached by a net of elements $(\gam_n g_n,g_n^{-1}{\gam'_n}^{-1})$.
Then the net of elements $(\lam^{-1}\gam_n g_n h^{-1},hg_n^{-1}{\gam'_n}^{-1}\lam')$
would approach $(e,e)$, which implies that $\lam^{-1}\gam_n{\gam'_n}^{-1}\lam'$ approaches
$e$, and hence is ultimately $e$, by discreteness; thus $\gam_n{\gam'_n}^{-1}$ is ultimately
$\lam{\lam'}^{-1}$, i.e.\ the net is ultimately in the fibre $(\lam,\lam')\cdot\check{\Del}_G$.

($\Rightarrow$)  If $\check{\Del}_G$ is of local synthesis for $G\times G$, then so too
is each fibre $(\gam,\gam')\cdot \check{\Del}_G=(\gam,e)\check{\Del}_G(e,{\gam'}^{-1})$.
We then appeal to Proposition \ref{prop:locsynthdisc} and (\ref{eq:open}) to obtain
necessity in (\ref{eq:injection}).

($\Leftarrow$) The anti-diagonal $\check{\Del}_G$ is an open fibre of
$q^{-1}(\check{\Del}_{G/\Gam})$, thanks to (\ref{eq:open}).
Then Proposition \ref{prop:locsynthdisc} delivers sufficiency in (\ref{eq:injection}).
\end{proof}

In practice we will use this last result with a central discrete subgroup.
Centrality seems to offer no meaningful simplification to the proof, however.

\section{Failure of local synthesis for the anti-diagonal}\label{sec:fivegroups}

\subsection{The basic strategy.}  We wish to show that for certain given
$2$ or $3$-dimensional non-abelian connected Lie groups $G$, that
its anti-diagonal fails to be a set of local synthesis for  $G\times G$ .
We shall exploit the density of the space of test functions $\fD(G)$ in
$A(G)$, as outlined in Section \ref{ssec:WALSAD}.  This implies that
any $T$ in $VN(G)\cong A(G)^*$ may be understood as a distribution,
i.e.\ $T\in\fD(G)^*$.

The following informs all of the choices made through the rest of this section.

\begin{lemma}\label{lem:basic}
Let $G$ be a connected Lie group with Lie algebra $\fg$ and $X\in\fg$ with $X\not=0$,
hence $(X,0)$ is an element of the Lie algebra $\fg\times\fg$ of $G\times G$.
Let $0\not=v\in L^1(G)$. If there exists $S=S_{X,v}$ in $VN(G\times G)$ for which
\[
S(u)= \int_G \partial_{(X,0)}u(g,g^{-1})v(g)\,dg
\]
whenever $u\in\fD(G\times G)$,  then $\check{\Del}_G$ is
not a set of local synthesis for $G\times G$.
\end{lemma}

\begin{proof}
We will verify that $S\in I^0_{G\times G}(\check{\Del}_G)^\perp\setminus
J_{G\times G}(\check{\Del}_G)^\perp$, where $K^\perp$ denotes the annihilator of the
subspace $K$.  This will show that $\wbar{I^0_{G\times G}(\check{\Del}_G)}\subsetneq
J_{G\times G}(\check{\Del}_G)$.

Let us see that $S\in I^0_{G\times G}(\check{\Del}_G)^\perp$.  Let
$u\in I^0_{G\times G}(\check{\Del}_G)$, and $\eps>0$. By the regularity of the algebra
$\fD(G\times G)$, there is $w$ in $\fD(G\times G)$ such that $w|_{\supp(u)}=1$
and $\supp(w)\cap\check{\Del}_G=\varnothing$.
Since $I^0_{G\times G}(\check{\Del}_G)\subseteq J^\fD_{G\times G}(\check{\Del}_G)$,
 there is $u'$ in $\fD(G\times G)\cap
I_{G\times G}(\check{\Del}_G)$ such that $\norm{u-u'}_A<\eps/\norm{w}_A$.  But then, since
$u=wu$, we have that $\norm{u-wu'}_A< \eps$. As $wu'\in  I^0_{G\times G}(\check{\Del}_G)$,
it is obvious that $\partial_{(X,0)}(wu')|_{\check{\Del}_G}=0$, so
$S(wu')=0$.  Hence $|S(u)|=|S(u-wu')|\leq \norm{S}_{VN}\eps$.  As $\eps>0$ may
be chosen arbitrarily, $S(u)=0$.

Let us now see that $S\not\in J_{G\times G}(\check{\Del}_G)^\perp$.
We consider general $x,y,z$ in $\fD(G)$, and let $w=xy\otimes z-x\otimes\check{y}z$, which
is an element of $\fD(G\times G)\cap J_{G\times G}(\check{\Del}_G)$.  But then
\[
S(w)=\int_G [\partial_X(xy)(g)-\partial_Xx(g)y(g)]z(g^{-1})v(g)\,dg
=-\int_G x(g)\partial_Xy(g)z(g^{-1})v(g)\;dg.
\]
We may choose $x,y,z$ for which $S(w)\not=0$.
\end{proof}

We shall require disintegration of the left regular representation of $G$ into
irreducible components.  For this purpose and to introduce notation,
we summarize the Plancherel theorem of \cite{tatsuuma}.
Our presentation is influenced by \cite[(7.50)]{folland}.
We purposely restrict the description, to fit our needs.
We let $\what{G}$ denote the space of (equivalence classes of) irreducible representations
of $G$, and accept the standard abuse of notation where we conflate an equivalence class with
one of its representatives.  If $u\in\fC_c(G)$ we let its Fourier transform be given at $\pi$ in $\what{G}$
by
\[
\hat{u}(\pi)=\int_G u(g)\pi(g)\;dg\in\fB(\fH_\pi)\text{, i.e.\ }
\langle \hat{u}(\pi)\xi|\eta\rangle=\int_G u(g)\langle \pi(g)\xi|\eta\rangle\,dg
\]
where $\xi,\eta\in\fH_\pi$, and $\fH_\pi$ is the space on which $\pi$ acts.

\begin{proposition}\label{prop:plancherel}
Suppose $G$ is a connected Lie group for which the kernel of the modular function,
$K=\ker\Del$, is type I, and $G$ acts on $\what{K}$ regularly in the sense
that there is a Borel cross-section for the space of orbits of $G$ on $\what{K}$.
Then $\what{G}$ contains
\begin{itemize}
\item a dense Borel subset $S(\what{G})$
of elements, each of which is the induced representation from some closed
subgroup of $\ker\Del$;

\item a Borel parametrization $b\mapsto \pi_b:B\to S(\what{G})$ and a Borel measure
$\mu$ on $B$; and

\item for each $b$ in $B$, a positive operator $\del_b$ on $\fH_{\pi_b}$, which
satisfies
\[
\pi_b(g)\del_b\pi_b(g^{-1})=\frac{1}{\Del(g)}\del_b
\]
\end{itemize}
for which the choice of the triple
$\left((\del_b)_{b\in B},(\pi_b)_{b\in B},\mu\right)$ is unique,
up to measure equivalence of $((\pi_b)_{b\in B},\mu)$, and
such that the Plancherel transform on $\fC_c(G)$, given  by
\[
u\mapsto(\what{u}(\pi_b)\del_b^{1/2})_{b\in B}
\]
extends to a unitary identifying
\[
L^2(G)\cong\int^\oplus_B\fH_{\pi_b}\otimes^2\wbar{\fH}_{\pi_b}\,d\mu(b).
\]
Each Hilbertian tensor product $\fH_{\pi_b}\otimes^2\wbar{\fH}_{\pi_b}$ is identified with the
space of Hilbert-Schmidt operators on $\fH_{\pi_b}$, above.
\end{proposition}

Hence the left regular
representation $\lam$ of $G$ admits disintegration up to unitary equivalence, and
quasi-equivalence respectively, as
\begin{equation}\label{eq:dislrr}
\lam\cong\int^\oplus_B \pi_b\otimes I_{\wbar{\fH}_{\pi_b}}\,d\mu(b)\quad\text{and}\quad
\lam\simeq\int^\oplus_B \pi_b \,d\mu(b).
\end{equation}
In (\ref{eq:dislrr}) we need to concern ourselves with only the equivalence class of $\mu$
in the relation of mutual absolute continuity of measures on $B$.
Furthermore, the left regular representation of $G\times G$
may now be represented by the quasi-equivalence
\begin{equation}\label{eq:dislrrp}
\lam\times\lam\simeq\int^\oplus_{B\times B} \pi_b\times \pi_{b'}\,d(\mu\times\mu)(b,b')
\end{equation}
where we use Kroenecker products of representations.  Hence if $v\in L^1(G)$,
then the operator on $A(G\times G)$ given by
$E_v(u)=\int_G u(g,g^{-1})v(g)\,dg$ may be represented by the operator field
\begin{equation}\label{eq:ev}
E_v(b,b')=\int_G v(g)\pi_b(g)\otimes\pi_{b'}(g^{-1})\,dg\quad\text{ for }b,b'\text{ in }B.
\end{equation}

If $G$ is unimodular, we set each $\del_b=I_{\fH_{\pi_b}}$.  When $G$ is not unimodular,
and $\pi_b$ is induced from a character $\chi$ of abelian subgroup $H$ of $\ker\Del$, then
$\fH_{\pi_b}$ may be identified with a completion of $\fF_{\pi_b}=\{f\in \fC(G):f(gh)=\wbar{\chi(h)}f(g)\}$,
and $\del_b$ with multiplication on the latter space by $\Del(\cdot)^{-1}$; compare with the description
in \cite[(7.49)]{folland}.

We also have for any $g$ in $G$, the Fourier inversion formula of \cite[Corollary 2]{tatsuuma}:
for $w=\langle{\lam(\cdot)}u|v\rangle$, where $u,v\in\fC_c(G)$, we have
\begin{equation}\label{eq:fourierinv}
w(g)=\int_B\tr(\pi_b(g^{-1})\hat{w}(\pi_b)\del_b)\,d\mu(b)
=\int_B\tr(\pi_b(g)\what{\check{w}}(\pi_b)\del_b)\,d\mu(b).
\end{equation}
By density of $\spn\lam(G)$ in $VN(G)$, and of each $\spn\pi_b(G)$ in $\fB(\fH_{\pi_b})$ (Schur's lemma),
we also have for any $T$ in $VN(G)$ the duality
formula
\begin{equation}\label{eq:fourierdual}
T(w)=\int_B\tr(T(b)\what{\check{w}}(\pi_b)\del_b)\,d\mu(b)
\end{equation}
where $T\simeq(T(b))_{b\in B}\in L^\infty(B,\mu;\fB(\fH_{\pi_b}))$.

Let $G$ be a connected Lie group and $\pi\in\what{G}$.
It is well known that for any $X$ in the Lie algebra $\fg$ of $G$,
and any $\pi$ in $\what{G}$, there is a dense subspace
$\fH_\pi^X$ of vectors $\xi$ for which
\begin{equation}\label{eq:dpi}
d\pi(X)\xi=\lim_{h\to 0}\frac{1}{h}[\pi(\exp(hX))-I]\xi
\end{equation}
exists.  For example $\fH_\pi^X\supseteq\fH_\pi^\fD=\spn\{\hat{u}(\pi)\xi:u\in\fD(G),\xi\in\fH_\pi\}$.
Thus $d\pi(X)$ is an (unbounded) operator on $\fH_\pi$.  Given the parameterization
$b\mapsto \pi_b$ above, it can be checked
that, as the limit of a measurable field of operators, $(d\pi_b(X))_{b\in B}$
is also measurable.

\begin{lemma}\label{lem:derfac}
Suppose that $G$ is a connected Lie group with Lie algebra $\fg$, and $X$ in $\fg$ is
such that, for each $b$ in $B$, there is a subspace
$\fF_b$ of $\fH^X_{\pi_b}$ which is dense in $\fH_{\pi_b}$, and for which
$\del_b^{-1}\fF_b\subseteq\fH_{\pi_b}^X$.   If
$T$  in $VN(G)$ satisfies that $(S(b))_{b\in B}=(T(b)d\pi_b(X))_{b\in B}$
is a bounded field of operators,  then for $u$ in $\fD(G)$ we have that
\[
S(u)=T(\partial_Xu).
\]
\end{lemma}

\begin{proof}  Recall that as in (\ref{eq:der}), the symbol $\partial_X$ denotes a derivative
on the right.
First we fix $\pi$.  If $\xi\in\del_b^{-1}\fF_b$ then we have for $u$ in $\fD(G)$ that
\begin{align*}
&\what{(\partial_X u)^\vee}(\pi_b)\del_b\xi
=\int_G \partial_X u(g^{-1})\pi_b(g)\del_b\xi\, dg \\
&\quad=\lim_{h\to\infty}\frac{1}{h}\left[\int_G u(g^{-1}\exp(hX))\pi_b(g)\del_b\xi\,dg
-\int_G u(g^{-1})\pi_b(g)\del_b\xi\,dg\right] \\
&\quad=\lim_{h\to\infty}\frac{1}{h}\left[\int_G u(g^{-1})\pi_b(\exp(hX)g)\del_b\xi\,dg
-\int_G u(g^{-1})\pi_b(g)\del_b\xi\,dg\right] \\
&\quad=\left(\lim_{h\to\infty}\frac{1}{h}[\pi_b(\exp(hX)-I]\right)\int_Gu(g^{-1})\pi_b(g)\del_b\, dg\xi  \\
&\quad= d\pi_b(X)\what{\check{u}}(\pi_b)\del_b\xi
\end{align*}
where in the second through fourth lines, the limit is understood in the weak sense, and we may use
dominated convergence theorem on associated scalar integrals.
Since $\del_b^{-1}\fF_b$ is dense in $\fH_{\pi_b}$,  the computation
above, coupled with (\ref{eq:fourierdual}), tells us that
\begin{align*}
S(u)&=\int_B \tr(T(b)d\pi_b(X)\what{\check{u}}(\pi_b)\del_b)\,d\mu(b) \\
&=\int_B \tr(T(b)\what{(\partial_Xu)^\vee}(\pi_b)\del_b)\,d\mu(b)=T(\partial_Xu).
\end{align*}
\end{proof}

We now will embark on using Lemma \ref{lem:derfac} on the group $G\times G$
and operator fields with fibres $d(\pi_b\times\pi_{b'})(X,0)=d\pi_b(X)\otimes I_{\fH_{b'}}$ ($X\in \fg$)
to verify the conditions of Lemma \ref{lem:basic}.

As a first illustration, let us apply these methods to
the special unitary group $\su(2)$.  We note that $\check{\Del}_{\su(2)}$ is a set
of non-synthesis is known (see \cite{forrestss1}).
We recall the well known fact that  $\su(2)$ admits
as its Lie algebra $\fsu(2)$, which may be identified with $2\times 2$-complex
matrices which are skew-Hermitian and trace zero.  The representation
theory of $\su(2)$ is well-known:  $\what{\su}(2)=\{\pi_n:n=0,1,2,\dots\}$,
each $\pi_n$ acts on a space of dimension $n+1$, and if $g\cong\diag(z,\bar{z})$
($\cong$ is similarity by conjugation in $\su(2)$) then $\pi_n(g)\cong\diag(z^n,z^{n-2},\dots,
z^{-n})$.

\begin{proposition}\label{prop:sutwo}
Let $X$ be any non-zero element of $\fsu(2)$.  Then there is an $S=S_{X,1}$ in $VN(\su(2))$
for which
\[
S(u)=\int_G \partial_{(X,0)}u(g,g^{-1})\,dg
\]
for $u$ in $\fD(\su(2)\times\su(2))$.
\end{proposition}

\begin{proof}
Since $X^*=-X$ and $\tr X=0$, there is $x$ in $\Ree$ for which we have equivalence
$X\cong\diag(ix,-ix)$,
hence $\exp(hX)\cong\diag(e^{ihx},e^{-ihx})$ in $\su(2)$.  It is immediate that
\[
d\pi_n(X)\cong\diag(inx,i(n-2)x,\dots,-inx).
\]
The Schur orthogonality relations immediately give, in the notation of (\ref{eq:ev}) with $v=1$, that
\[
E_1(n,n')=\int_{\su(2)}\pi_n(g)\otimes\pi_{n'}(g)^*\,dg=\begin{cases} \frac{1}{n+1}T_n &\text{if }n=n' \\
0 &\text{otherwise}\end{cases}
\]
where $T_n$ is an $(n+1)^2\times(n+1)^2$ permutation matrix with respect to some basis.
It is then obvious that
\[
\norm{E_1(n,n')(d\pi_n(X)\otimes I)}\leq |x|
\]
for each $n,n'$.  We appeal to Lemma \ref{lem:derfac}.
\end{proof}

\subsection{Two unimodular groups}
We let $E=\Cee\rtimes\Tee$ be the 3-dimensional Euclidean motion group
with multiplication and inversion given by
\[
(x,z)(x',z')=(x+zx',zz')\quad\text{and}\quad (x,z)^{-1}=(-\bar{z}x,\bar{z}).
\]
This group is unimodular with Haar integral the same as that on $\Cee\times\Tee$.

The following data are well known; see, for example \cite[IV]{sugiura}.  The additive
group $\Cee=\Ree^2$ admits a real inner product $(x,x')\mapsto x\cdot x'=\re x\re x'+\im x\im x'$.
For each $a$ in $\Cee\setminus\{0\}$ we obtain an irreducible unitary representation $\pi_a$
by inducing from the character $\chi_a$ ($\chi_a(x)=e^{-ix\cdot a}$).  Then
$\pi_a\cong\pi_b$ if and only if $|a|=|b|$.  Hence we parameterize this family by $r$ in $(0,\infty)$.
Each representation is given
\[
\pi_r:E\to\fU(L^2(\Tee)),\quad \pi_r(x,z)\xi(w)=e^{-ix\cdot(rw)}\xi(\bar{z}w).
\]
Then the disintegration formula  (\ref{eq:dislrr}) takes the form
\[
\lam\simeq\int_{(0,\infty)}^\oplus \pi_r\,r\,dr\simeq\int_{(0,\infty)}^\oplus \pi_r\,dr
\]
where the middle formula is with respect to the group's Plancherel measure which
is mutually absolutely equivalent to Lebesgue measure on $(0,\infty)$.

We recall that $E$ has Lie algebra
\[
\fe=\langle T,X_1,X_2 : [T,X_1]=X_2,[T,X_2]=-X_1,[X_1,X_2]=0\rangle
\]
where $\exp(hX_1)=(h,1)$, $\exp( hX_2)=(ih,1)$ and $\exp(hT)=(0,e^{ih})$.

We also consider the Sobolev-type space
\[
H^{2,1}(\Ree^2)=\{v:\Ree^2\to\Cee\;|\;v,\partial_x^2v\in L^1(\Ree^2),
\text{ for all }x\in\Ree^2\}
\]
where $\partial_xv(y)=\left.\frac{d}{dt}v(y+tx)\right|_{t=0}$, and the second order derivatives
may be considered in the distributional sense.

\begin{theorem}\label{theo:euclmot}
Let $\til{v}\in H^{2,1}(\Cee)=H^{2,1}(\Ree^2)$, and $v$ in $L^1(E)$ be given by $v(x,z)=\til{v}(x)$.
Let $X\in\spn\{X_1,X_2\}$.   Then there is an $S=S_{X,v}$ in $VN(E)$
for which
\[
S(u)=\int_E \partial_{(X,0)}u(g,g^{-1})v(g)\,dg
\]
for $u$ in $\fD(E\times E)$.
\end{theorem}

\begin{proof}
First, for $\xi$ in $L^2(\Tee)$, (\ref{eq:ev}) provides that
\begin{align*}
E_v(r,r')\xi(w,w') &=\int_\Tee\int_\Cee \til{v}(x)e^{-ix\cdot(rw-r'zw')}\xi(\bar{z}w,zw')\,dx\,dz \\
&=  \int_\Tee2\pi \what{\til{v}}(rw-r'zw')\xi(\bar{z}w,zw')\,dz
\end{align*}
where $\what{\til{v}}$ is the Fourier transform of $\til{v}$ on $\Cee=\Ree^2$.  Consider
the unitary $U$ on $L^2(\Tee^2)$ given by $U\xi(w,w')=\xi(w,ww')$, which has
adjoint given by $U^*\xi(w,w')=\xi(w,\bar{w}w')$.  We have
\begin{align*}
U^*&E_v(r,r')U\xi(w,w')=E_v(r,r')U\xi(w,\bar{w}w')  \\
&= 2\pi\int_\Tee \what{\til{v}}(rw-r'z\bar{w}w')\xi(\bar{z}w,w')\,dz
= 2\pi\int_\Tee \what{\til{v}}(rw-r'zw')\xi(\bar{z},w')\,dz
\end{align*}
The fact that we choose $\til{v}$ from $H^{2,1}(\Cee)$ allows that there is a constant $C_v$,
such that for any $y$ in $\Cee$ we have
\[
|\what{\til{v}}(y)|\leq\frac{C_v}{1+|y|^2}.
\]
In fact we could choose $C_v=\|\what{\til{v}}\|_\infty+\|\what{L\til{v}}\|_\infty$,
where $L=\partial_{x_1}^2+\partial_{x_2}^2$ is the Laplacian on $\Ree^2\cong
\Cee$.  Hence we estimate
\begin{align*}
&\norm{U^*E_v(r,r')U\xi(w,w')}_2^2  \leq \int_{\Tee^2}\left[2\pi C_v \int_\Tee \frac{|\xi(\bar{z},w')|}{1+|rw-r'zw'|^2}\,dz\right]^2\,d(w,w') \\
&\quad\leq (2\pi C_v)^2\int_\Tee\int_\Tee \left[\int_\Tee \frac{dz}{1+|r-r'zw'\bar{w}|^2}\right]
\left[\int_\Tee \frac{|\xi(\bar{z},w')|^2\,dz}{1+|r-r'w'z\bar{w}|^2}\right]\,dw\,dw' \\
&\quad\leq (2\pi C_v)^2\left[\int_\Tee \frac{dz}{1+|r-r'z|^2}\right]^2\norm{\xi}_2^2
\end{align*}
where we have used the Cauchy-Schwarz inequality in the second line, and Tonelli's
theorem, in the third. Hence
\begin{equation}\label{eq:nerr}
\norm{E_v(r,r')}=\norm{U^*E_v(r,r')U}\leq \int_\Tee \frac{2\pi C_v\, dz}{1+|r-r'z|^2}
\text{ for }r,r'>0.
\end{equation}

Now we consider for each $r>0$ the operator $d\pi_r(X)$.  If
$\xi\in L^2(\Tee)$ then for $(x,1)=\exp X$ we have
\[
d\pi_r(X)\xi(w)
=\lim_{h\to 0}\frac{1}{h}[e^{-i(hx)\cdot(rw)}-1]\xi(w)=-irx\cdot w\xi(w)
\]
where convergence is uniform in $w$. Hence by (\ref{eq:dpi}) $d\pi_r(X)$ is the multiplication
operator by $w\mapsto -irx\cdot w$.  In particular $d\pi_r(X)$ is bounded
with $\norm{d\pi_r(X)}\leq r|x|$.  Combining with (\ref{eq:nerr}) we see that
\[
\norm{E_v(r,r')(d\pi_r(X)\otimes I)}\leq  \int_\Tee \frac{2\pi C_vr\,dz}{1+|r-r'z|^2}
=\int_0^{2\pi}\frac{C_vr\,dt}{1+r^2+{r'}^2-2rr'\cos t}.
\]
However, using either methods of complex analysis, or
the table of integrals \cite[2.553-3]{gradshteynr}, we
obtain that the latter integral is equal to the first expression in the elementary estimate
\[
\frac{C_vr}{2\sqrt{(1+r^2+{r'}^2)^2-4(rr')^2}}\leq \frac{C_vr}{2\sqrt{1+2(r^2+{r'}^2)}}
\]
which is clearly uniformly bounded in $r$ and $r'$.
Hence, by Lemma \ref{lem:derfac}, we are done.
\end{proof}

Now we consider the reduced Heisenberg group $\Hee^r=(\Ree\times\Tee)\rtimes\Ree$ with
multiplication and inversion given by
\[
(y,z,x)(y',z',x')=(y+y',zz'e^{ixy'},x+x')\quad\text{and}\quad
(y,z,x)^{-1}=(-y,\bar{z}e^{ixy},-x).
\]
We identify the centre of $\Hee^r$ with $\Tee$.
The group is unimodular, and its Haar integral is the same as that on the product group
$\Ree\times\Tee\times \Ree$.
All of the infinite-dimensional irreducible representations are known to be
obtained by inducing from the characters $\chi_{0,n}$ on the normal subgroup $\Ree\times\Tee$,
$\chi_{0,n}(y,z)=\bar{z}^n$, for $n\in\Zee\setminus\{0\}$.
This follows, for example, from \cite[4.38]{kaniutht} and the fact that $\Hee^r$ is a quotient of the usual
Heisenberg group. For each $n$ in $\Zee\setminus\{0\}$, the representation is given by
\[
\pi_n:\Hee^r\to L^2(\Ree),\quad\pi_n(y,z,x)\xi(t)=\bar{z}^ne^{-inty}\xi(t-x).
\]
The left regular representation admits a decomposition
\begin{equation}\label{eq:heislrr}
\lam\simeq\bigoplus_{n\in\Zee}\pi_n
\end{equation}
where $\pi_0:\Hee^r\to L^2(\Ree^2)$ is, effectively, the left regular representation of
$\Hee^r/\Tee$.  Indeed, if we let $\fH_n=\{\xi\in L^2(\Hee^r):\xi(g(0,z,0))=\bar{z}^n\xi(g)\text{ for
a.e.\ }g\text{ in }\Hee^r\}$, then $\lam_n=\lam(\cdot)|_{\fH_n}$ has $\lam_n(0,z,0)=\bar{z}^n I$,
and is hence quasi-equivalent to $\pi_n$ by the Stone-von Neumann theorem (see, for example,
\cite[(6.49)]{folland}).

We note that the Lie algebra of $\Hee^r$ is given by
\[
\fh=\langle X,Y,Z:[X,Y]=Z,[Z,X]=0=[Y,Z]\rangle
\]
where $\exp(hX)=(0,1,h),\;\exp(hY)=(h,1,0)$ and $\exp(hZ)=(0,e^{ih},0)$.

\begin{theorem}\label{theo:heisenberg}
Let $v_1,v_2$ in $L^1(\Ree)$ be so that $v_1$ is (essentially) bounded and $\hat{v}_2\in L^1(\Ree)$,
and set $v(y,z,x)=v_1(x)v_2(y)$.  Then there is $S=S_{Z,v}$ in $VN(\Hee^r)$ for which
\[
S(u)=\int_{\Hee^r}\partial_Zu(g,g^{-1})v(g)\,dg
\]
for $u\in\fD(\Hee^r\times\Hee^r)$.
\end{theorem}

\begin{proof}
We allow (\ref{eq:heislrr}) to substitute for (\ref{eq:dislrr}) and we obtain a likewise
decomposition for $\lam\times\lam$.  Thus, using the appropriate analogue of (\ref{eq:ev})
we compute for $n,n'$ in $\Zee\setminus\{0\}$ that for $\xi$ in $L^2(\Ree^2)\cong
L^2(\Ree)\otimes^2 L^2(\Ree)$ that
\begin{align*}
E_v(n,n')\xi(t,t')&=
\int_\Ree\int_\Tee\int_\Ree v_1(x)v_2(y)\bar{z}^ne^{-inty}{\wbar{ze^{ixy}}}^{n'}e^{in't'y}\xi(t-x,t'+x)\,dy\,dz\,dx \\
&=\del_{n,n'}\int_\Ree\int_\Ree v_1(x)v_2(y)e^{-in(x+t-t')y}\xi(t-x,t'+x)\,dy\,dx \\
&=\sqrt{2\pi}\del_{n,n'}\int_\Ree v_1(x)\hat{v}_2(n(x+t-t'))\xi(t-x,t'+x)\,dx
\end{align*}
where $\del_{n,n'}$ is the Kroenecker delta symbol.  Likewise $E_v(n,0)=0$ for $n$ in
$\Zee\setminus\{0\}$.  We consider the unitary on $L^2(\Ree)$ given by
$U\xi(t,t')=\xi(t,t'+t)$.  We have for $n$ in $\Zee\setminus\{0\}$ that
\begin{align*}
U^*E_v(n,n)U\xi(t,t')&=E_v(n,n)U\xi(t,t'-t) \\
&=\sqrt{2\pi}\int_\Ree v_1(x)\hat{v}_2(n(x+2t-t'))\xi(t-x,t')\,dx \\
&=\sqrt{2\pi}\int_\Ree v_1(t-x)\hat{v}_2(n(-x+3t-t'))\xi(x,t')\,dx
\end{align*}
Hence we compute
\begin{align*}
\norm{U^*E_v(n,n)U\xi}_2^2
&\leq \int_\Ree\int_\Ree 2\pi \left[\int_\Ree |v_1(t-x)\hat{v}_2(n(-x+3t-t'))\xi(x,t')|\,dx\right]^2\,dt\,dt' \\
&\leq 2\pi \int_\Ree\int_\Ree\left[\int_\Ree |v_1(t-x)\hat{v}_2(n(-x+3t-t'))|\,dx\right] \\
&\qquad\quad\times  \left[\int_\Ree |v_1(t-x)\hat{v}_2(n(-x+3t-t'))||\xi(x,t')|^2\,dx\right]\,dt\,dt' \\
&\leq 2\pi \norm{v_1}^2_\infty\frac{\norm{\hat{v}_2}^2_1}{3n^2}\norm{\xi}^2_2
\end{align*}
where we have used the Cauchy-Schwarz inequality for the second inequality and Tonelli's theorem,
a H\"{o}lder inequality and a change of variables for the third. In summary
\begin{equation}\label{eq:ennnorm}
\norm{E_v(n,n')}
\leq \del_{n,n'}\frac{\sqrt{2\pi}}{\sqrt{3}n}\norm{v_1}_\infty\norm{\hat{v}_2}_1\text{, for }(n,n')
\text{ in }(\Zee\setminus\{0\})\times\Zee.
\end{equation}
It is trivial to see that
\[
d\pi_n(Z)=-in I \text{ for }n\text{ in }\Zee.
\]
Combining  with (\ref{eq:ennnorm}) we see that
\[
\sup_{(n,n')\in \Zee^2}\norm{E(n,n')(d\pi_n(Z)\otimes I)}<\infty
\]
and we may hence appeal to Lemma \ref{lem:derfac}.
\end{proof}


\subsection{Two non-unimodular groups}
We shall consider a class of groups we call {\it Gr\'{e}laud's groups}.
Fix a parameter $\theta>0$ and for $s$ in $\Ree$ let
\[
\tau(s)=\exp s\begin{bmatrix} 1 & -\theta \\ \theta & 1 \end{bmatrix}
=e^s\varrho(s)\text{, where }\varrho(s)
=\begin{bmatrix} \cos s\theta & -\sin s\theta \\  \sin s\theta& \cos s\theta  \end{bmatrix}.
\]
We now let $G_\theta=\Ree^2\rtimes_\tau\Ree$ with multiplication given by
\[
(x,s)(x',s')=(x+\tau(s)x',s+s')\text{, hence }(x,s)^{-1}=(-\tau(-s)x,-s).
\]
Notice that we have $\det\tau(s)=e^{2s}$ from which we
get left Haar integral,  for $u$ in $\fC_c(G_\theta)$, and modular function
\[
\int_G u(x,s)\,d(x,s)=\int_{\Ree^2}\int_\Ree u(x,s)e^{-2s}\,ds\,dx\quad\text{and}\quad
\Del(x,s)=e^{2s}.
\]
If $y\in\Ree^2$, with associated character $\chi_y(x)=e^{-ix\cdot y}$ on $\Ree^2$, we get
induced representation
\[
\pi_y:G_\theta\to \fU(L^2(\Ree)),\quad
\pi_y(x,s)=e^{-i(\tau(-t)x)\cdot y}\xi(t-s)
=e^{-ie^{-t}x\cdot (\varrho(t)y)}\xi(t-s).
\]
Notice that if $y'=\varrho(t')y$, then $\pi_y\cong\pi_{y'}$ via the intertwiner
$\rho(t')$, where $\rho$ is the left regular representation on $L^2(\Ree)$.
Hence we may parameterize these representations by the unit sphere
$\Ess^1$.  Furthermore, by \cite[7.35]{kaniutht}, the family $\{\pi_y\}_{y\in\Ess^1}$
is a dense, compact subset of $\what{G}_\theta$.

To learn the disintegration of the left regular representation (\ref{eq:dislrr}),
we will have to obtain the Plancherel formula for this group.  Since we do not
know of a reference for this, we compute it ourselves.  We use the notation
of Proposition \ref{prop:plancherel}.

\begin{proposition}
The Plancherel decomposition of $L^2(G_\theta)$ is given by
\[
L^2(G_\theta)\cong\int_{\Ess^1}^\oplus L^2(\Ree)_y\otimes^2 \wbar{L^2(\Ree)}_y\,d\nu(y)
\]
where each $L^2(\Ree)_y$ is a copy of $L^2(\Ree)$, $\nu$ is the unique
rotationally invariant probability measure on $\Ess^1$ and we have
for $\xi$ in $L^2(\Ree)_y$, $\del_y\xi(t)=\Del(t)^{-1}\xi(t)=e^{-2t}\xi(t)$ for a.e.\ $t$.
Hence we obtain
\[
\lam\simeq \int_{\Ess^1}^\oplus \pi_y\;d\nu(y)
\]
\end{proposition}

\begin{proof}
We will simply verify that the choices above satisfy Proposition \ref{prop:plancherel}.
Let for $y$ in $\Ess^1$ and $u$ in $\fC_c(G_\theta)$
and $\xi$ in $L^2(\Ree)$
\begin{align*}
\hat{u}(\pi_y)\del_y^{1/2}\xi(t)&=\int_G u(g)\pi_y(g)(\del_y^{1/2}\xi)(t)\,dg \\
&=\int_\Ree\int_{\Ree^2}u(x,s) e^{-ie^{-t}x\cdot(\varrho(t)y)}e^{s-t}\xi(t-s)\,dx\,ds \\
&=\int_\Ree 2\pi\hat{u}^1(e^{-t}\varrho(t)y,t-s)e^{-s}\xi(s)\,ds
\end{align*}
where $\hat{u}^1$ is the partial Fourier transform in the $\Ree^2$-variable.  We note the well-known fact
that the Hilbert-Schmidt norm $\norm{\hat{u}(\pi_y)}_2$ is given in terms of
the kernel function as $\left(\int_\Ree\int_\Ree|2\pi\hat{u}^1(e^{-t}\varrho(t)y,s+t)|^2\,e^{-2s}ds\,dt\right)^{1/2}$.  We
then see that
\begin{align*}
\int_{\Ess^1} \norm{\hat{u}(\pi_y)\del_y}_2^2&\,d\nu(y)
=\int_{\Ess^1} \int_\Ree\int_\Ree|2\pi\hat{u}^1(e^{-t}\varrho(t)y,t-s)|^2\,e^{-2s}ds\,dt\,d\nu(y) \\
&=\int_{\Ess^1} \int_\Ree\int_\Ree|2\pi\hat{u}^1(e^{-t}y,s)|^2\,e^{2(t-s)}ds\,dt\,d\nu(y) \\
&=\int_{\Ree^2}\int_\Ree|2\pi\hat{u}^1(z,s)|^2\,e^{-2s}ds\,dz
=\int_\Ree\int_{\Ree^2}|u(x,s)|^2\, dx\, e^{-2s}ds
=\norm{u}_2^2.
\end{align*}
where we have used the invariance of the chosen measures on $\Ess^1$ and $\Ree$,
then an obvious change of variables in $\Ree^2$, and finally the Plancherel formula in $\Ree^2$.
\end{proof}

We note that $G_\theta$ admits the Lie algebra
\[
\fg_\theta=\langle T,X_1,X_2 : [T,X_1]=X_1-\theta X_2,
[T,X_2]=\theta X_1+X_2,[X_1,X_2]=0\rangle
\]
where $\exp(hX_j)=(he_j,0)$ for $j=1,2$ ($(e_1,e_2)$
is the standard basis for $\Ree^2$) and $\exp(hT)=(0,h)$.

\begin{theorem}\label{theo:grelaud}
Let $v_1$ in $L^1(\Ree^2)$ be so that $\hat{v}_1\in A_c(\Ree^2)$,
let $v_2\in \fC_c(\Ree)$, and then let $v(x,s)=v_1(x)v_2(s)$.
Let $X\in\spn\{X_1,X_2\}$.  Then there is an $S=S_{X,v}$ in $VN(G_\theta)$
for which
\[
S(u)=\int_{G_\theta} \partial_{(X,0)}u(g,g^{-1})v(g)\,dg
\]
for $u$ in $\fD(G_\theta\times G_\theta)$.
\end{theorem}

\begin{proof}
Our assumptions on $v_1$ and $v_2$ ensure that $v$ is integrable.
We use (\ref{eq:ev}) to compute for $\xi$ in $L^2(\Ree^2)\cong L^2(\Ree)\otimes^2 L^2(\Ree)$ that
\begin{align*}
E_v(y,y')&\xi(t,t') \\
&=\int_\Ree \int_{\Ree^2}v_1(x)v_2(s)e^{-i(\tau(-t)x)\cdot y+(-\tau(-t'-s)x)\cdot y' }\xi(t-s,t'+s)
\,dx\,e^{-2s}ds \\
&=\int_\Ree 2\pi\hat{v}_1(e^{-t}\varrho(t)y-e^{-t'-s}\varrho(t'+s)y')v_2(s)e^{-2s}\xi(t-s,t'+s)\,ds.
\end{align*}
Letting, now, $U\xi(t,t')=\xi(t,t'+t)$ we compute
\begin{align}
U^*E_v&(y,y')U\xi(t,t')=E_v(y,y')U\xi(t,t'-t) \notag \\
&=2\pi \int_\Ree \hat{v}_1(e^{-t}\varrho(t)y-e^{-t'+t-s}\varrho(t'-t+s)y')v_2(s)e^{-2s}\xi(t-s,t')\,ds \notag \\
&=2\pi\int_\Ree \hat{v}_1(e^{-t}\varrho(t)y-e^{-t'+s}\varrho(t'+s)y')v_2(t-s)e^{2(s-t)}\xi(s,t')\,ds
\label{eq:Guevu}
\end{align}

Let us now compute $d\pi_y(X)$.  Let $(x,0)=\exp X$, and for $\xi\in\fC_c(\Ree)$ we have
\begin{equation}\label{eq:Gdpi}
d\pi_y(X)\xi(t)=\lim_{h\to 0 }\frac{1}{h}[e^{-ie^{-t}(hx)\cdot(\varrho(t)y)}-1]\xi(t)
=-ie^{-t}x\cdot(\varrho(t)y)\xi(t)
\end{equation}
where convergence is uniform on compact sets.  Hence
$d\pi_y(X)$ admits $\fF_y=\fC_c(\Ree)$ in its domain, and by (\ref{eq:dpi}),
and is given by pointwise multiplication by $t\mapsto -ie^{-t}x\cdot(\varrho(t)y)$.
Notice that $\del_y^{-1}\fF_y\subset\fF_y$.  Hence we may appeal to Lemma \ref{lem:derfac}
and it suffices to see that the field of operators
\begin{equation}\label{eq:Gopf}
(E_v(y,y')(d\pi_y(X)\otimes I))_{y,y'\in\Ess^1}\text{ is bounded}
\end{equation}
to gain our conclusion.  It is clear that $U^*(d\pi_y(X)\otimes I)U=d\pi_y(X)\otimes I$, and
it suffices to work with the field $(U^*E_v(y,y')U(d\pi_y(X)\otimes I))_{y,y'\in\Ess^1}$.
We combine (\ref{eq:Guevu}) and (\ref{eq:Gdpi}) to see that for $\xi$ in $L^2(\Ree^2)$ we have
\begin{align*}
U^*&E_v(y,y')U(d\pi_y(X)\otimes I)\xi(t,t') \\
&=2\pi\int_\Ree \hat{v}_1(e^{-t}\varrho(t)y-e^{-t'+s}\varrho(t'+s)y')v_2(t-s)
e^{2(s-t)}(d\pi_y(X)\otimes I)\xi(s,t')\,ds \\
&=-2\pi i\int_\Ree \hat{v}_1(e^{-t}\varrho(t)y-e^{-t'+s}\varrho(t'+s)y')v_2(t-s)
e^{s-2t}(x\cdot(\varrho(s)y))\xi(s,t')\,ds.
\end{align*}
Our assumptions on $v_1$ allow us to find a non-increasing $\vphi$ in $\fC_c([0,\infty))$
for which
\[
2\pi|\hat{v}_1(y)|\leq\varphi(|y|).
\]
Let us observe that $|e^{-t}\varrho(t)y-e^{-t'+s}\varrho(t'+s)y'|\geq |e^{-t}-e^{-t'+s}|$.  Now with
an application of Cauchy-Schwarz inequality we obtain
\begin{align*}
&\norm{U^*E_v(y,y')U(d\pi_y(X)\otimes I)\xi}_2^2 \\
&\quad\leq |x|^2\int_{\Ree^2}  \left[\int_\Ree\varphi(|e^{-t}-e^{-t'+s}|)|v_2(t-s)|
e^{s-2t}|\xi(s,t')|\,ds\right]^2\,d(t,t') \\
&\quad\leq |x|^2\int_{\Ree^2}\left[\int_\Ree \varphi(|e^{-t}-e^{-t'+s}|)e^{-t}|v_2(t-s)|e^{s-t}\,ds\right] \\
&\quad\qquad\quad
\times\left[\int_\Ree \varphi(|e^{-t}-e^{-t'+s}|)e^{-t}|v_2(t-s)|e^{s-t}|\xi(s,t')|^2\,ds\right]\,d(t,t').
\end{align*}
Since $e^{-t}\leq|e^{-t}-e^{-t'+s}|+e^{-t'+s}$ we have
\begin{align*}
\int_\Ree &\varphi(|e^{-t}-e^{-t'+s}|)e^{-t}|v_2(t-s)|e^{s-t}\,ds  \\
&\leq \int_\Ree \varphi(|e^{-t}-e^{-t'+s}|)|e^{-t}-e^{-t'+s}||v_2(t-s)|e^{s-t}\,ds  \\
&\quad\qquad\quad +
\int_\Ree \varphi(|e^{-t}-e^{-t'+s}|)e^{-t'+s}|v_2(t-s)|e^{s-t}\,ds  \\
&\leq\norm{\vphi\iota}_\infty\|\del^{-1/2} v_2\|_1+\norm{\vphi}_1\|\del^{-1/2} v_2\|_\infty
\end{align*}
where $\iota(r)=r$ for $r\geq 0$ and $\del v_2(s)=e^{2s}v_2(s)$.
It then follows that
\begin{align*}
&\norm{U^*E_v(y,y')U(d\pi_y(X)\otimes I)\xi}_2^2 \\
&\quad\leq |x|^2\left(\norm{\vphi\iota}_\infty\norm{\del v_2}_1+\norm{\vphi}_1\norm{\del v_2}_\infty\right) \\
&\quad\qquad\quad
\times\int_\Ree\int_\Ree\left[\int_\Ree \varphi(|e^{-t}-e^{-t'+s}|)e^{-t}|v_2(t-s)|e^{s-t}\,dt
\right]|\xi(s,t')|^2\,ds\,dt' \\
&\quad\leq|x|^2\left(\norm{\vphi\iota}_\infty\|\del^{-1/2} v_2\|_1+\norm{\vphi}_1\|\del^{-1/2} v_2\|_\infty\right)
\norm{\vphi}_1\|\del^{-1/2} v_2\|_\infty\norm{\xi}_2^2.
\end{align*}
Hence we have verified (\ref{eq:Gopf}).
\end{proof}

We now consider the real affine motion group, also known as the ``$ax+b$ group",
$F=\Ree\rtimes\Ree$ with product and inverse given by
\[
(b,a)(b',a')=(b+e^ab',a+a')\quad\text{and}\quad
(b,a)^{-1}=(-e^{-a}b,-a).
\]
The group has left Haar integral and modular function given by
\[
\int_Fu(b,a)d(b,a)=\int_\Ree\int_\Ree u(b,a)\,e^{-a}da\,db\quad
\text{and}\quad\Del(b,a)=e^a.
\]
It is extremely well known, see, for example \cite{khalil} or \cite[7.6-2]{folland}, that the only
inequivalent infinite-dimensional irreducible representations are given by $\pi_{\pm}$, which are gained
by inducing characters of positive/negative index
from the normal subgroup $B=\{(b,0);B\in\Ree\}\cong\Ree$.
Explicitly, for $\xi$ in $L^2(\Ree)$ we have
\[
\pi_{\pm}(b,a)\xi(t)=e^{\mp ie^{-t}b}\xi(t-a).
\]
Furthermore, we get a quasi-equivalence
\begin{equation}\label{eq:affmotlrr}
\lam\simeq \pi_+\oplus \pi_-.
\end{equation}
We recall that the Lie algebra for $F$ is given by
\[
\ff=\langle X,Y : [X,Y]=Y\rangle
\]
where $\exp(hX)=(0,h)$ and $\exp(hY)=(h,0)$.

\begin{theorem}\label{theo:affmot}
Let $v_1$ in $L^1(\Ree)$ be so $\what{v}_1\in A_c(\Ree)$,  let
$v_2\in\fC_c(\Ree)$, and set $v(b,a)=v_1(b)v_2(a)$.
Then there is  $S=S_{Y,v}$ in $VN(F)$ such that
\[
S(u)=\int_F \partial_{(Y,0)}u(g,g^{-1})v(g)\,dg
\]
for  $u$ in $\fD(F\times F)$.
\end{theorem}

\begin{proof}
The details of this proof are similar to those in the proof of Theorem \ref{theo:grelaud}. 
Indeed, we obtain
for $\sig,\sig'$ in $\{\pm\}$ and $\xi$ in $L^2(\Ree^2)$ that
\[
E_v(\sig,\sig')\xi(t,t')
=\sqrt{2\pi}\int_\Ree \hat{v}_1(\sig e^{-t}-\sig' e^{-a-t'})v_2(a)e^{-a}\xi(t-a,t'+a)\,da
\]
and if $U\xi(t,t')=\xi(t,t'+t)$ we have
\[
U^*E(\sig,\sig')U\xi(t,t')=
\sqrt{2\pi}\int_\Ree \hat{v}_1(\sig e^{-t}-\sig' e^{a-t'})v_2(t-a)e^{a-t}\xi(a,t')\,da.
\]
We compute for $\xi$ in $\fF_\pm=\fC_c(\Ree)$ the derivative
\[
d\pi_{\pm}(Y)\xi(t)=\lim_{g\to 0}\frac{1}{h}[e^{\mp ie^{-t}h}-I]\xi(t)=\mp i e^{-t}\xi(t).
\]
The operators from the Plancherel formula (Proposition \ref{prop:plancherel}) are
given by $\del_\pm\xi(t)=e^{-t}\xi(t)$, and we see that $\del_\pm\fF_\pm\subseteq\fF_\pm$.
Hence we may appeal to Lemma \ref{lem:derfac}, and it suffices to see that each of the operators
\begin{equation}\label{eq:evpm}
E_v(\sig,\sig')(d\pi_\sig\otimes I)\text{ is bounded for each }\sig,\sig'\text{ in }\{\pm\}.
\end{equation}
Since $U^*(d\pi_\sig\otimes I)U=d\pi_\sig\otimes I$ it suffices compute the norms of the operators
$U^*E_v(\sig,\sig')U(d\pi_\sig\otimes I)$.  We first observe that
\[
U^*E_v(\sig,\sig')U(d\pi_\sig\otimes I)\xi(t,t')
=\sqrt{2\pi}i\sig\int_\Ree \hat{v}_1(\sig e^{-t}-\sig' e^{a-t'})v_2(t-a) e^{-t}\xi(a,t')\,da.
\]
We then note that $e^{-t}\leq |\sig e^{-t}-\sig' e^{a-t'}|+e^{a-t'}$ and observe that
\begin{align*}
\int_\Ree |\hat{v}_1(\sig e^{-t}-&\sig' e^{a-t'})v_2(t-a)| e^{-t}\,da \\
&\leq \int_\Ree |\hat{v}_1(\sig e^{-t}-\sig' e^{a-t'})||\sig e^{-t}-\sig' e^{a-t'}||v_2(t-a)| \,da \\
&\qquad\quad+\int_\Ree |\hat{v}_1(\sig e^{-t}-\sig' e^{a-t'})|e^{a-t'}|v_2(t-a)| \,da \\
&\leq \norm{\hat{v}_1\alp}_\infty\norm{v_2}_1+\norm{\hat{v}_1}_1\norm{v_2}_\infty.
\end{align*}
where $\alp(t)=|t|$ for $t$ in $\Ree$.
Hence by our usual technique we see that
\begin{align*}
&\norm{U^*E_v(\sig,\sig')U(d\pi_\sig\otimes I)\xi}_2^2 \\
&\quad\leq 2\pi
\left( \norm{\hat{v}_1\alp}_\infty\norm{v_2}_1+\norm{\hat{v}_1}_1\norm{v_2}_\infty\right) \\
&\quad\qquad\quad \times \int_\Ree\int_\Ree\int_\Ree
|\hat{v}_1(\sig e^{-t}-\sig' e^{a-t'})|e^{-t}|v_2(t-a)| |\xi(a,t')|^2\,dt\,da\,dt' \\
&\quad\leq 2\pi\left( \norm{\hat{v}_1\alp}_\infty\norm{v_2}_1+\norm{\hat{v}_1}_1\norm{v_2}_\infty\right)
\norm{\hat{v}_1}_1\norm{v_2}_\infty\norm{\xi}_2^2.
\end{align*}
Hence (\ref{eq:evpm}) is verified.
\end{proof}


\subsection{Remarks}  We note that the failure of weak amenability of $A(G)$, for
$G$ either $F$ or $\Hee^r$, is shown in \cite{choig}; but this does not automatically imply failure of 
local synthesis.  

The one aspect in common in  the strategies employed for each of the five (classes of) 
groups $G$ above is that the Lie derivative is always taken from a direction which is trivial
in any abelian quotient.  We suspect that to do otherwise would entail that for an abelian quotient,
$G/N$, we would find that $\check{\Del}_{G/N}$ would be a set of non-synthesis, which
would contradict Theorem \ref{theo:wassad} as $A(G/N)\cong L^1(\widehat{G/N})$
is amenable.

Our choice of Lie derivatives $X$, in forming operators $S_{X,v}$, 
shares a property in common with the choices made in \cite{johnson}
for $\mathrm{SO}(3)$, and \cite{choig,choig1} for $F$, $\Hee^r$ and the Heisenberg group $\Hee$.
Of course, there is an enormous gulf between exploiting these for a failure of spectral synthesis calculation
and showing that these derivatives may be used to build non-trivial elements of $H^1(A(G),VN(G))$.
It is plausible that for each of our four semi-direct product groups $G=N\rtimes A$, above, with the associated
Lie algebra $\fn\rtimes\fa$, that the space of bounded derivations from $A(G)$ to $VN(G)$ is isomorphic to 
the Lie algebra $\fn$.

\section{Weak amenability of Fourier algebras}\label{sec:final}

The following is well-known.  We include a proof for convenience of non-specialists.

\begin{proposition}\label{prop:basiclie}
{\bf (i)}  Let $\fg$ be a non-abelian real Lie algebra.  Then $\fg$ contains either
$\fsu(2)$ or a non-abelian solvable algebra $\fm$.

{\bf (ii)}  Let $\fm$ be a non-abelian solvable real Lie algebra.  Then $\fm$ contains one of
the following  Lie algebras:
\begin{align*}
\ff&=\langle X,Y : [X,Y]=Y\rangle\text{ (affine motion)} \\
\fe&=\langle T,X_1,X_2 : [T,X_1]=X_2,[T,X_2]=-X_1,[X_1,X_2]=0\rangle
\text{ (Euclidean motion)} \\
\fg_\theta&=\langle T,X_1,X_2 : [T,X_1]=X_1-\theta X_2,[T,X_2]=\theta X_1+X_2,[X_1,X_2]=0\rangle, \;
\theta>0, \\
&\phantom{mm}\text{ (Gr\'{e}laud), or} \\
\fh&=\langle X,Y,Z:[X,Y]=Z,[X,Z]=0=[Y,Z]\rangle\text{ (Heisenberg).}
\end{align*}
\end{proposition}

\begin{proof}  Every concept and fact from Lie theory which is used in this proof
is well-known and may be found in \cite{hilgertnB}, for example.

(i) Let $\fr$ denote the solvable radical ideal of $\fg$.  If $\fr$ is non-abelian, we set $\fm=\fr$.
If $\fr$ is abelian, we use the Levi-decomposition $\fg=\fs+\fr$, where $\fs$ is semisimple.
For $\fs$ we have the Iwasawa decomposition $\fs=\fk+\fa+\fn$.
If $\fk$ is non-abelian, then $\fk$ contains $\fsu(2)$ and we are done.  Otherwise $\fm=\fa+\fn$
is non-abelian and solvable; indeed, see the construction in \cite[\S 13.3]{hilgertnB}.
[In \cite[Proposition 5.3]{choig} a further refinement shows that $\fm\supseteq \ff$, in this case.]

(ii)  We let $\fm\supsetneq\fm'\supsetneq\dots\supsetneq\fm^{(d-1)}\supsetneq\{0\}$ be the derived series,
so $\fm^{(d-2)}$ is non abelian.  For simplicity we assume $\fm=\fm^{(d-2)}$; thus we
now have derived series $\fm\supsetneq\fm'\supsetneq\fm''=\{0\}$.

Suppose, first
there is an element $S$ of $\fm$ for which $\ad S:\fm\to\fm$ admits a non-zero eigenvalue
$a+ib$, $a,b\in\Ree$.  If $b=0$, then there is an eigenvector $Y$ for $\ad S$, i.e.\
$[S,Y]=aY$.  Put $X=\frac{1}{a}S$ and we get $\ff=\spn\{X,Y\}$.  Otherwise, we appeal to the
real Jordan decomposition of $\ad S$, which allows us to find elements $Y_1,Y_2$ of $\fm$ for which
\[
[S,Y_1]=aY_1-bY_2\quad\text{and}\quad [S,Y_2]=bY_1+aY_2.
\]
If $a=0$, we let $T=\frac{1}{b}S$ and we notice immediately that $Y_1=X_1,Y_2=X_2\in \fm'$, so
$[Y_1,Y_2]=0$.  Thus we get $\fe=\spn\{T,Y_1,Y_2\}$.  Finally, if $ab\not=0$,
we let $\theta=b/a$.  Let
$T=\frac{1}{a}S$ and $X_j=Y_j$ ($j=1,2$) if $ab>0$; and let $T=\frac{1}{b}S$,
$X_1=Y_2$, $X_2=Y_1$ if $ab<0$.  It is straightforward to check that
$X_1,X_2\in\fm'$ so $[X_1,X_2]=0$.  Hence, $\fg_\theta=\spn\{T,X_1,X_2\}$, in this case.

If no $S$ in $\fm$ has that $\ad S$ admits a non-zero eigenvalue, then each $\ad S$ is nilpotent,
and hence by Engel's theorem $\fm$ is nilpotent.  We consider the central series
$\fm\supsetneq \fC(\fm)\supsetneq\dots\supsetneq \fC^n(\fm)\supsetneq\{0\}$.
Any $Z$ in $\fC^n(\fm)$ is hence central and we further consider such $Z$ for which
there are $X,Y$ in $\fm$ (one of which is in $\fC^{n-1}(\fm)$) for which $Z=[X,Y]$.
Then $\fh=\spn\{X,Y,Z\}$.
\end{proof}

\begin{theorem}\label{theo:main}
Let $G$ be a connected Lie group.  Then the following are equivalent:

{\bf (i)} $A(G)$ is weakly amenable,

{\bf (ii)} the anti-diagonal $\check{\Del}_G$ is of local synthesis for $G\times G$, and

{\bf (iii)} $G$ is abelian.
\end{theorem}

\begin{proof}
The result (i) $\Rightarrow$ (ii) is  Theorem \ref{theo:wassad}, while (iii) $\Rightarrow$ (i)
is well-known, i.e.\ $A(G)\cong L^1(\what{G})$, in this case.  Hence it remains
to prove that (ii) $\Rightarrow$ (iii).  We will assume that $G$ is non-abelian, and show that
$\check{\Del}_G$ is not of local synthesis for $G\times G$.

All Lie theoretic concepts and nomenclature, used in this proof, may be found in \cite{hilgertnB}.
We will show that when $G$ is non-abelian, then $\check{\Del}_G$ fails to be of local synthesis
for $G\times G$.  Thanks to Theorem \ref{theo:discnorm},
and the fact that any $G\cong\widetilde{G}/\Gam$ for a simply connected connected Lie
group $\widetilde{G}$ and a discrete central subgroup  $\Gam$ of $\widetilde{G}$, we may
restrict ourselves to the case of simply connected groups.

Let $G$ be simply connected with Lie algebra $\fg$.    The Levi decomposition $\fg=\fs+\fr$ begets
respective integral subgroups $S$ and $R$ of $G$, with $S$ semisimple and $R$ normal, and
hence, essentially by the smooth splitting theorem, each is closed and simply connected.  
Given a Lie subalgebra $\fm$ of $\fg$, in particular one of the algebras arising from 
Proposition \ref{prop:basiclie}, above, we let $M$ be the integral subgroup generated by $\fm$.
If $\fm=\fsu(2)$, so $\fm\subset\fs$, then 
$M\cong\su(2)/C$, where $C$ is a subgroup of the centre $\{\pm 1\}$, 
and hence $M$ is closed.  If $\fm\subset\fs$ and is non-abelian 
and solvable, then $M\subseteq AN$ (Iwasawa decomposition of $S$); or if $\fm\subseteq\fr$, then 
$M\subseteq R$.   Any integral subgroup of a simply connected solvable group is closed and simply 
connected.  Thus, corresponding to each of the solvable 
Lie algebras $\ff$, $\fe$, $\fg_\theta$ or $\fh$, as 
obtained in Proposition \ref{prop:basiclie}, above, we see that $M$ is closed and isomorphic to one of
the affine motion group $F$, the simply connected cover $\widetilde{E}$ of the Euclidean 
motion group $E$, G\'{e}laud's group $G_\theta$, or the Heisenberg group $\Hee$.   

Proposition \ref{prop:sutwo}, Theorem \ref{theo:euclmot}, Theorem \ref{theo:heisenberg},
Theorem \ref{theo:grelaud}, and Theorem \ref{theo:affmot}, 
and then Lemma \ref{lem:basic}, tells us that $\check{\Del}_H$ fails to be a set of (local) synthesis
for $H\times H$, for each of $H=\su(2),\,E,\,\Hee^r,\,G_\theta$ and $F$.
Then, Theorem \ref{theo:discnorm} tells us the same for $H=\su(2)/C$, $\widetilde{E}$ and $\Hee$.
Hence applying Theorem \ref{theo:restriction} to the subgroup $M$ of $G$, above, we obtain (ii).  
\end{proof}

Let $G$ be a locally compact group.
The following is an immediate consequence of the restriction theorem for closed
subgroups $H$, i.e.\ $R_H(A(G))=A(H)$ (see the comments before Theorem \ref{theo:restriction});
and the fact that any quotient of a commutative weakly amenable
Banach algebra is again weakly amenable.

\begin{corollary}\label{cor:main}
Let $G$ be a locally compact group. If $A(G)$ is weakly amenable,
then any connected Lie subgroup of $G$ is abelian.
\end{corollary}

\begin{example}
Let us consider an example of a non-abelian solvable connected group which contains no
non-abelian closed connected Lie subgroups.
We consider a group related to the simply connected covering group, 
$\widetilde{E}=\Cee\rtimes\Ree$, of the Euclidean motion group.

We let $\Ree^{ap}$ denote the almost periodic compactification
of the real line, and for each $r$ in $\Ree$ we let 
$\chi_r:\Ree^{ap}\to\Tee$ be the unique continuous character extending
$t\mapsto e^{itr}:\Ree\to\Tee$. We observe that the dual group is given by 
$\widehat{\Ree^{ap}}=\{\chi_r:r\in\Ree\}$ and is isomorphic to the discretized reals, $\Ree_d$.
If $K$ is a closed connected Lie subgroup of $\Ree^{ap}$, then $K\cong\Tee^n$, where
$n=0,1,2,\dots$, so $\widehat{K}\cong\Zee^n$.  On the other hand
$\what{K}$ is a quotient of $\Ree_d$, and hence is either trivial or divisible.  Thus $K=\{1\}$, and
we conclude that $\Ree^{ap}$ contains no non-trivial closed Lie subgroups. 

Now let $\wbar{E}=\Cee\rtimes\Ree^{ap}$ with multiplication given by
\[
(x,\zeta)(x',\zeta')=(x+\chi_1(\zeta)x',\zeta\zeta').
\]
We may think of this group as a partial compactification of $\widetilde{E}$ along its quotient subgroup.
We note that $q(x,\zeta)=(x,\chi_1(\zeta))$ defines a quotient homomomorphism from $\wbar{E}$
onto $E$.
For notation convenience we identify $\Ree^{ap}\cong\{0\}\rtimes\Ree^{ap}$ in $\wbar{E}$, 
and $\Tee\cong\{0\}\rtimes\Tee$ in $E$, and we identify $\Cee$ with $\Cee\rtimes\{1\}$ in
either $\wbar{E}$, or in $E$, as should be clear by context.  Notice that $q|_\Cee$ is 
a homeomorphism.

We now show that $\wbar{E}$ admits no non-abelian closed connected Lie subgroups.
We first note that the proper Lie subalgebras of $\fe$ are either one-dimensional, the two-dimensional
ideal $\fn=\Ree X_1+\Ree X_2$, or $\{0\}$.  Since every closed connected subgroup
of $E$ corresponds to a Lie subalgebra of $\fe$, we find that the one-dimensional
closed connected subgroups are of the form
\begin{align*}
M_{S+X}&=\exp(\Ree(S+X))=\{((z-1)x,z):z\in\Tee\} \\
\text{and }L_X&=\exp(\Ree X)=\{(tx,1):t\in\Ree\}
\end{align*}
where $X\in\fn$ with $\exp X=(x,1)$, while the only two-dimensional closed connected subgroup is
$\Cee$.  Taking $q^{-1}(M)$ for each subgroup $M$ of $E$, listed above, we get
\[
\wbar{M}_{S+X}=\{((\chi_1(\zeta)-1)x,\zeta):\zeta\in\Ree^{ap}\},\;
L_X\rtimes\ker\chi_1\text{ and }\Cee\rtimes\ker\chi_1.
\]
Let $H$ be a proper closed connected subgroup of $\wbar{E}$ and $H_0=H\cap\Cee$.  
Notice that in $E$, $q(H_0)=q(H)\cap\Cee\cong H_0$.  If $H_0=\Cee$ then
for $(x,\zeta)$ in $H$ we have $(0,\zeta)=(-x,1)(x,\zeta)\in H\cap\Ree^{ap}$,
so the image of $H$ under the quotient map $(x,\zeta)\mapsto\zeta$ is closed, and hence either
the trivial group $\{1\}$ or is non-Lie.  If $H_0\subsetneq\Cee$, then 
$\wbar{q(H)}\cap\Cee\cong H_0\subsetneq\Cee$, and it follows that $H$ is a subgroup
of one of the abelian groups $\wbar{M}_{S+X}$, $L_X\rtimes\ker\chi_1$ or $\Ree^{ap}$, 
and hence is abelian.

At present, we are aware of no method for determining if $A(\wbar{E})$ is weakly amenable.
\end{example}

{\bf Acknowledgement.}  We are grateful to S{\o}ren Knudby for pointing out a flaw in an earlier version of the proof of Theorem \ref{theo:main}, and for guidance in correcting it. We also thank Yemon Choi for pointing out an error in a proposition which purported to reduce the problem to the case of pro-solvable $G_e$. We realized there was a serious gap in the proof and have deleted the proposition from the final version of the paper.

Addresses:
\linebreak
 {\sc
Department of Mathematical Sciences  and Research Institute of Mathematics, Seoul National University,
Gwanak-ro 1, Gwanak-gu, Seoul 08826, Republic of Korea

\medskip\noindent
Institut \'{E}lie Cartan de Lorraine,
Universit\'{e} de Lorraine -- Metz,
B\^{a}timent A, Ile du Saulcy, F-57045 Metz, France

\medskip\noindent
Department of Mathematics and Statistics, University of Saskatchewan,
Room 142 McLean Hall, 106 Wiggins Road
Saskatoon, SK, S7N 5E6, Canada

\medskip\noindent
Department of Pure Mathematics, University of Waterloo,
Waterloo, ON, N2L 3G1, Canada}

\medskip
Email-adresses:
\linebreak
{\tt hunheelee@snu.ac.kr}
\linebreak {\tt  ludwig@univ-metz.fr}
\linebreak {\tt samei@math.usask.ca}
\linebreak {\tt nspronk@uwaterloo.ca}

\end{document}